\documentclass[12pt]{article}
\usepackage{amsmath,amsthm, amssymb, amsfonts,accents,nccmath}
\usepackage{enumitem}
\usepackage{array}
\usepackage{lipsum}
\usepackage{mathtools}
\usepackage[toc,page]{appendix}
\usepackage[english]{babel}
\usepackage{dsfont}
\usepackage{booktabs}
\usepackage[linesnumbered,commentsnumbered,ruled,vlined]{algorithm2e}
\usepackage{tikz}
\usetikzlibrary{decorations.pathreplacing,calligraphy}
\usepackage{tikz-network}
\usetikzlibrary{hobby}
\usepackage[utf8]{inputenc}
\usepackage[english]{babel}
\usepackage{float}

\usepackage{caption}
\usepackage{subcaption}
\usetikzlibrary{arrows}

\usepackage{cite}

\newcommand\restr[2]{{
  \left.\kern-\nulldelimiterspace 
  #1 
  \vphantom{\big|} 
  \right|_{#2} 
  }}

\usepackage{hyperref}
\hypersetup{
    colorlinks=true,
    linkcolor=blue,
    filecolor=darkmagenta,      
    urlcolor=cyan,
    citecolor=red
} 
\makeatletter
\tikzset{
	dot diameter/.store in=\dot@diameter,
	dot diameter=2pt,
	dot spacing/.store in=\dot@spacing,
	dot spacing=9pt,
	dots/.style={
		line width=\dot@diameter,
		line cap=round,
		dash pattern=on 0pt off \dot@spacing
	}
}
\makeatother

\makeatletter
\def\BState{\State\hskip-\ALG@thistlm}
\makeatother
\numberwithin{equation}{section}

\newtheorem{theorem}{Theorem}[section]
\newtheorem{cor}[theorem]{Corollary}
\newtheorem{lem}[theorem]{Lemma}
\newtheorem{prop}[theorem]{Proposition}

\newtheorem{defn}[theorem]{Definition}
\newtheorem{rem}[theorem]{Remark}
\newtheorem{ex}[theorem]{Example}

\def \bd{\begin{defn}}
	\def \ed{\end{defn}}
\def \bt{\begin{theorem}}
	\def \et{\end{theorem}}
\def \bl{\begin{lem}}
	\def \el{\end{lem}}
\def \bc{\begin{cor}}
	\def \ec{\end{cor}}
\def \be{\begin{equation}}
	\def \ee{\end{equation}}
\def \ba{\begin{array}}
	\def \ea{\end{array}}
\def\bp{\begin{prop}}
	\def\ep{\end{prop}}
\def \bx{\begin{ex}}
	\def \ex{\end{ex}}
\def\bxa{\begin{ex}\rm}
	\def\exa{\end{ex}}
\def \br{\begin{rem}}
	\def \er{\end{rem}}
\def \bdsc{\begin{description}}
	\def \edsc{\end{description}}
\def \bn{\begin{case}}
	\def \en{\end{case}}
\def \bfig{\begin{figure}}
	\def \efig{\end{figure}}
\def \bpic{\begin{picture}}
	\def \epic{\end{picture}}
\def \bnn{\begin{note}}
	\def \enn{\end{note}}\def
\bpr{\begin{problem}}
	\def\epr{\end{problem}}
\def \pf{\begin{proof}}
	\def \qe{\end{proof}}

\newcommand{\lbd}{\lambda}

\title{On the minimum spectral radius of unicyclic graphs with a given matching number}
\author{Joyentanuj Das\footnote{Department of Mathematics, College of Engineering and Technology, SRM Institute of Science and Technology, Kattankulathur, Chennai, 603203, India. \indent  Email: joyentanuj@gmail.com,  joyentad@srmist.edu.in} \quad and \quad Debabrota Mondal\footnote{Department of Mathematics, Indian Institute of Technology Guwahati, Guwahati, 781039, India. \indent  Email: debabrota@iitg.ac.in}}

\date{}

\begin{document}

\maketitle

\begin{abstract}
A matching $M$ in a graph $G = (V, E)$ is a set of edges such that no two edges in $M$ share a common vertex. A matching with maximum cardinality is called a maximum matching and its cardinality is the matching number $\gamma(G)$. The spectral radius of $G$ is the maximum absolute eigenvalue of its adjacency matrix. This article addresses the Brualdi-Solheid problem--the determination of extremal spectral radii within specific graph classes--for the class $\mathcal{U}_{n,\gamma}$ of simple connected unicyclic graphs on $n$ vertices with matching number $\gamma$. We specifically characterize all graphs that achieve the minimum spectral radius in $\mathcal{U}_{n,\gamma}$ for matching numbers $\gamma \in \left\{ 1, 2, 3, \lfloor \frac{n}{2} \rfloor \right\}$.
\end{abstract}

\noindent {\sc\textsl{Keywords}:}  Adjacency matrix, Matching number, Spectral radius, Unicyclic graphs.

\noindent {\textbf{MSC}:}   05C50, 05C15, 15A18

\section{Introduction and Preliminaries}\label{sec:intro}
Let $G=(V,E)$  be a finite, simple, connected graph with $V$ as the set of vertices and $E$ as the set of edges in $G$. We write $u\sim v$ to indicate that the vertices $u$ and $v$ are adjacent in $G$. The degree of the vertex $v$, denoted by $d_G(v)$, equals the number of vertices in $V$ that are adjacent to $v$. The neighborhood of a vertex $v \in V$ is the set $N_G(v) = \{u \in V(G)\colon u \sim v \}$. Thus, $d_G(v) = |N_G(v)|$. A pendant vertex is a vertex of degree one and a quasi-pendant vertex is a vertex with degree at least $2$ that is adjacent to a pendant vertex. For $u,v \in V$, the distance between $u$ and $v$ is defined to be the length of the shortest path between $u$ and $v$ and it is denoted by $d(u,v)$.

A graph $H=(V_1,E_1)$ is said to be a subgraph of $G=(V_2,E_2)$ if $V_1 \subset V_2$ and $E_1 \subset E_2$. For any subset $S \subset V_2$, a subgraph $H$ of $G$ is said to be an induced subgraph with vertex set $S$, if $H$ is a maximal subgraph of $G$ with vertex set $V_1=S$. We write $|S|$ to denote the cardinality of the set $S$. For a vertex $v \in V(G)$, we use $G-v$ to denote the subgraph of $G$ obtained by deleting the vertex $v$ along with all the edges that have $v$ as one of its endpoints. We will use the notation $G \cong H$ to denote that two graphs $G$ and $H$ are isomorphic.

Let $G=(V,E)$ be a graph. For $u,v \in V$, the adjacency matrix of the graph  $G$ is,    ${A}(G) = [a_{uv}]$, where $a_{uv}= 1 $  if $u\sim v$ and $0$ otherwise. By eigenvalues, eigenvectors, and characteristic polynomial of $G$, we mean those of $A(G)$. We use the notation $\phi_\lbd(G)$ to denote the characteristic polynomial of $G$.  For any column vector $\mathbf{x}$, if $x_u$ represents the entry of $\mathbf{x}$ corresponding to the vertex $u\in V$, then 
$$\mathbf{x}^{\top} {A}(G)\mathbf{x}=  2\sum_{u\sim w}x_{u}x_{w}, $$
where $\mathbf{x}^{\top}$ represent  the transpose of  $\mathbf{x}$. 

For a connected graph $G$ on $n \geq 2$ vertices, by the Perron--Frobenius theorem, the spectral radius  $\rho(G)$ of ${A}(G)$ is a simple positive eigenvalue and the associated eigenvector is entry-wise positive (for details see~\cite{Bapat}). We will refer to such an eigenvector as a Perron vector of $G$.  Now we state a few known results on spectral radius useful in our subsequent calculations. By the Min-max theorem, we have
\begin{equation}\label{eqn:eq_sp_rd}
\rho(G) = \max_{\mathbf{x}\neq \mathbf{0}}\dfrac{\mathbf{x}^{\top} {A}(G)\mathbf{x}}{\mathbf{x}^{\top} \mathbf{x}}= \max_{\mathbf{x}\neq \mathbf{0}}\dfrac{2\sum_{u\sim w}x_{u}x_{w}}{\sum_{u\in V}x_{u}^2}.
\end{equation}
Furthermore, in Eqn.~\eqref{eqn:eq_sp_rd}, the maximum is attained if and only if $\mathbf{x}$ is a Perron vector of $\rho(G).$

The following result compares the spectral radius of a graph and its proper subgraph. The proof can be found in \rm\cite{Li}.
\bl\label{lem:subgraph}{\rm\cite{Li}} Let $G$ be a graph and $H$ be a proper subgraph of $G$. Then $\rho(H) < \rho(G)$. \el


Since the characteristic polynomials of graphs are real-rooted monic polynomials with integer coefficients, the following result is important to compare the spectral radius of two graphs.
\bl\label{poly} Let $f(\lbd)$ and $g(\lbd)$ be two real-rooted monic polynomials, and $\mu_1(f)$ and $\mu_1(g)$ be their largest roots, respectively. Then the following assertions hold:
\bdsc
\item{\rm(i)} If $f(\lbd)> g(\lbd)$ for $\lbd\geq\mu_1(g)$, then $\mu_1(f)<\mu_1(g)$.
\item{\rm(ii)}  Suppose that $f(\lbd)-g(\lbd)=c(\lbd-\mu)$ such that $c>0$. If $\mu>\mu_1(g)$, then $\mu_1(f)>\mu_1(g)$.
\edsc
\el
\pf Proof of (i): On the contrary, assume that $\mu_1(f)\geq\mu_1(g)$. Since $g(\lbd)$ is a monic polynomial and $\mu_1(g)$ is its largest root, we have $g(\mu_1(f))\geq 0$. On the other hand, by assumption, $f(\lbd)>g(\lbd)$ for all $\lbd\geq \mu_1(g)$, and hence, $g(\mu_1(f))<f(\mu_1(f))=0$. This is clearly a contradiction. Therefore, $\mu_1(f)<\mu_1(g)$.

Proof of (ii): On the contrary, assume that $\mu_1(f)\leq\mu_1(g)$. Then $f(\mu_1(g))\geq0$. Again, we have given that $\mu_1(g)<\mu$. Note that $f(\lbd)< g(\lbd)$ for $\lbd<\mu$. Thus, $f(\mu_1(g))<0$, a contradiction.
\qe


Let $G = (V, E)$ be a graph. We say two edges are independent if they do not share a common vertex. A {\em matching} $M$ in  $G$ is a collection of independent edges. A matching $M$ is said to be a {\em maximum matching} if it possesses the largest cardinality among all matchings in $G$; specifically, $M$ is maximum if $|M| \ge |M'|$ for every matching $M' \subseteq E$. This maximum cardinality is termed the matching number of $G$ and is denoted by $\gamma(G)$. It is a fundamental property that $\gamma(G) \le \frac{n}{2}$, where $n = |V|$, with the upper bound achieved if and only if $G$ contains a perfect matching. The {\em girth} of a graph $G$ is the length of the minimum cycle. If $G$ is a unicyclic graph, then the girth of $G$ is the length of its unique cycle.

A walk $[v_1, v_2, \dots, v_k] (k \geq 2)$ of a graph $G$ with length $k-1$ is called an {\em internal path}, if these $k$ vertices are distinct (except possibly $v_1 = v_k$), $d(v_1) \geq 3, d(v_k) \geq 3$ and $d(v_2) = \dots = d(v_{k-1}) = 2$ (unless $k = 2$). An edge that belongs to some internal path of $G$ is called an {\em internal edge} of $G$. By $P_n$, $C_n$, and $S_n$, we denote the path, cycle and star graph on $n$ vertices, respectively. Let $W_n$ be the tree obtained from $P_{n-2}$ $(n \geq 3)$ by adding two pendant vertices to both of its quasi-pendant vertices. Let $G$ be a graph and $G_{uv}$ be the graph obtained from $G$ by the subdivision of the edge $u v$ (that is, by deleting the edge $uv$, adding a new vertex $w$ and two new edges $uw$ and $wv$). The following result compares the spectral radii of $G$ and $G_{uv}$.

\begin{lem}\label{lem-sub}{\rm\cite{Hoffman}}
Assume that $G \not\cong W_n$ and $uv$ is an edge on an internal path of $G$. Let $G_{uv}$ be the graph obtained from $G$ by the subdivision of the edge $uv$. Then $\rho(G_{uv}) < \rho(G)$.
\end{lem}

\begin{lem}\label{lem:rho}{\rm\cite{Wu}}
	Let $u$ and $v$ be two distinct vertices in a connected graph $G$ and $\{v_i \mid i = 1, 2, \dots, s\} \subseteq N_G(v) \setminus N_G(u)$. Let $\mathbf{x} = (x_{v_1}, x_{v_2}, \dots, x_{v_n})^{\top}$ be the Perron vector of $G$, where $x_{v_i}$ corresponds to $v_i$ {\rm(}$1 \leq i \leq n${\rm )}. Let $G^*$ be the graph obtained from $G$ by deleting the edges $vv_i$ for $1 \le i \le s$ and adding the edges $uv_i$ for $1 \le i \le s$. If $x_u \geq x_v$, then $\rho(G) < \rho(G^*)$.
\end{lem}

\begin{lem}{\rm \cite{Lin}}\label{lem-pendant}
	Let $G$ be a connected graph and $u$ and $v$ be two vertices of $G$ such that $G - u \cong G - v$. Define $G_{s,t}$ {\rm (}$s \geq t \geq 1${\rm )} to be the graph obtained from $G$ by attaching $s$ pendent vertices at $u$ and $t$ pendant vertices at $v$. Then $\rho(G_{s+1,t-1}) > \rho(G_{s,t})$. In particular, $\rho(G_{s+t,0}) > \rho(G_{s+t-1,1}) > \dots > \rho \left( G_{\left\lfloor \frac{s+t}{2} \right\rfloor,\left\lceil \frac{s+t}{2} \right\rceil}\right)$.
\end{lem}

In~\cite{Brualdi}, Brualdi and Solheid proposed the following general problem, which has become one of the classic problems of spectral graph theory: Given a set $\mathcal{G}$ of graphs on $n$ vertices, determine the maximum and minimum value of the spectral radii of graphs in $\mathcal{G}$ and characterize the extremal graphs that attain these maximum and minimum values.

The maximization part of this problem has been extensively studied across various graph and tree classes bounded by specific invariants including the domination number~\cite{Chen, Stevanovic1}, number of cut vertices~\cite{Berman}, clique number~\cite{Feng1}, chromatic number~\cite{Nikiforov}, matching number~\cite{Feng2, Guo1}, diameter~\cite{Guo2, Dam1}, independence number~\cite{Das1, Ji, Stevanovic2}, dissociation number~\cite{Das2}, and zero forcing number~\cite{Das3}. However, the minimization counterpart has received significantly less attention. To date, minimization results have only been established for invariants such as the independence number~\cite{Lou, Xu} and diameter~\cite{Dam2}.

Recently, Pirzada et al.~\cite{Pirzada} demonstrated that the graph minimizing the spectral radius among all $n$-vertex graphs with a fixed matching number is necessarily a tree. Moreover, for each $\gamma \in \left\{1,2,3,\left\lfloor\frac{n}{2}\right\rfloor-1,\left\lfloor\frac{n}{2}\right\rfloor\right\}$, they characterized the unique tree on $n$ vertices with matching number $\gamma$ that attains the minimum spectral radius. Motivated by these recent developments, we consider the following problem and provide a complete answer.

\textbf{Problem:} Let $\mathcal{U}_{n,\gamma}$ be the class of all connected unicyclic graphs on $n$ vertices with matching number $\gamma$. Find the extremal graph that attains the minimum spectral radius among the class of unicyclic graphs with $\gamma \in \left\{ 1, 2, 3, \lfloor \frac{n}{2} \rfloor \right\}$.


\section{Minimal spectral radius of unicyclic graphs with $\gamma \in \left\{ 1, 2, \lfloor \frac{n}{2} \rfloor \right\}$}
For $\gamma = 1$ and $\gamma = \lfloor \frac{n}{2} \rfloor$, the only possible unicyclic graph is the triangle or $K_3$ and the cycle on $n$ vertices or $C_n$, respectively. Thus, it remains to determine the graph with minimum spectral radius in the class $\mathcal{U}_{n,2}$. Let $G_1=[v_1,v_2,v_3,v_4,v_1]$ be the cycle graph on $4$ vertices and $a$ and $b$ be two nonnegative integers. Let $G_1(a,b)$ be the graph obtained from $G_1$ by attaching $a$ pendant vertices to $v_2$ and $b$ pendant vertices to $v_4$ (See Figure \ref{fig:gamma-2-4}). Note that $\gamma(G_1(a,b))=2$. We now prove that among all graphs in $\mathcal{U}_{n,2}$, the graph $G_1\left( \left\lfloor \frac{n-4}{2} \right\rfloor, \left\lceil \frac{n-4}{2} \right\rceil\right)$ attains the minimum spectral radius.

\begin{figure}[h]
	\centering
	\begin{tikzpicture}[scale=1]
  \coordinate (A) at (0.00, 0.00);
  \coordinate (B) at (2.00, 0.00);
  \coordinate (C) at (0.00, 2.00);
  \coordinate (D) at (2.00, 2.00);
  \coordinate (G) at (-1.00, 1.50);
  \coordinate (H) at (-1.00, 2.50);
  \coordinate (E) at (3.00, 0.50);
  \coordinate (F) at (3.00, -0.50);

  \draw (0,0) node[anchor= east] {$v_1$};
  \draw (0,2) node[anchor= south] {$v_2$};
  \draw (2,0) node[anchor= north] {$v_4$};
  \draw (2,2) node[anchor= west] {$v_3$};

  \draw (C) -- (D);
  \draw (C) -- (A);
  \draw (A) -- (B);
  \draw (D) -- (B);
  \draw (H) -- (C);
  \draw (G) -- (C);
  \draw (B) -- (E);
  \draw (B) -- (F);
  \fill[black] (A) circle (2pt);
  \fill[black] (B) circle (2pt);
  \fill[black] (C) circle (2pt);
  \fill[black] (D) circle (2pt);
  \fill[black] (G) circle (2pt);
  \fill[black] (H) circle (2pt);
  \fill[black] (E) circle (2pt);
  \fill[black] (F) circle (2pt);

  \draw (-1.8,2.7) node[anchor=north west] {$a\Biggl\{$};  
  \draw (3,0.7) node[anchor=north west] {$\Biggl\}b$};

  \draw[loosely dotted, line width=1.5pt] (3.00, 0.50) -- (3.00, -0.50);
  \draw[loosely dotted, line width=1.5pt] (-1.00, 1.50) -- (-1.00, 2.50);
\end{tikzpicture}
	\caption{$G_1(a,b)$}
	\label{fig:gamma-2-4}
\end{figure}

\begin{theorem}\label{thm-gamma-2}
	Let $n \ge 6$ be an integer. For any graph $G \in \mathcal{U}_{n,2}$, we have $\rho(G) \ge \rho \left(G_1 \left( \lfloor \frac{n-4}{2} \rfloor, \lceil \frac{n-4}{2} \rceil\right) \right)$ and the equality holds if and only if $G \cong G_1 \left( \left\lfloor \frac{n-4}{2} \right\rfloor, \left\lceil \frac{n-4}{2} \right\rceil\right)$.
\end{theorem}

\begin{proof}
    Let $G$ be a graph in $\mathcal{U}_{n,2}$. If the length of the unique cycle is greater than $4$, then $\gamma(G)>2$ (since $n \geq 6$). Thus, the girth of $G$ is at most $4$. Note that if the cycle length is $4$, then $G\cong G_1(a,b)$ for some nonnegative integers $a$ and $b$ such that $a+b=n-4$. Since $G_1-v_2\cong G_1-v_4$, Lemma \ref{lem-pendant} implies that $$\rho(G_1(a,b))\geq \rho\left(G_1\left(\left\lfloor\frac{a+b}{2}\right\rfloor,\left\lceil\frac{a+b}{2}\right\rceil\right)\right)=\rho\left(G_1\left(\left\lfloor\frac{n-4}{2}\right\rfloor,\left\lceil\frac{n-4}{2}\right\rceil\right)\right),$$
    with equality if and only if $G\cong G_1\left(\left\lfloor\frac{n-4}{2}\right\rfloor,\left\lceil\frac{n-4}{2}\right\rceil\right)$.

Now, suppose that the girth of $G$ is $3$ and $C_3=v_1v_2v_3v_1$ denotes the unique cycle of $G$. Let $S$ be the set consisting of the vertices of $G$ except $v_1,v_2,v_3$. Since $\gamma(G)=2$, either every vertex in $S$ is a pendant vertex, or $S$ contains exactly one quasi-pendant vertex. Thus, $G$ is isomorphic to one of the following two structures as shown in Figure \ref{mat_2_cycle_3}. Note that the graph $U_1$ is obtained from $C_3$ by attaching $a>0$ pendant vertices to $v_1$ and $n-3-a$ pendant vertices to $v_3$. The graph $U_2$ is obtained from $C_2$ by first attaching a pendant vertex $v_4$ to $v_3$ and then attaching $n-4$ pendant vertices to the vertex $v_4$. 

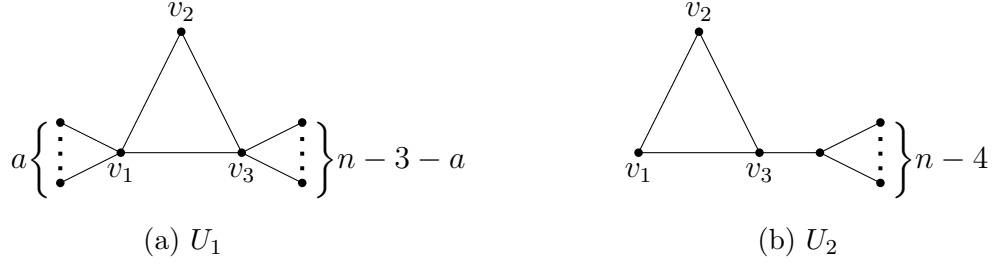
\begin{figure}[ht]
    \centering
    \begin{subfigure}[b]{0.3\textwidth}
        \centering
        \begin{tikzpicture}[scale=0.8]
  \coordinate (A) at (0.00, 0.00);
  \coordinate (B) at (2.00, 0.00);
  \coordinate (C) at (1.00, 2.00);
  \coordinate (D) at (-1.00, 0.50);
  \coordinate (E) at (-1.00, -0.50);
  \coordinate (F) at (3.00, 0.50);
  \coordinate (G) at (3.00, -0.50);

  \draw (0,0) node[anchor= north] {$v_1$};
  \draw (1,2) node[anchor= south] {$v_2$};
  \draw (2,0) node[anchor= north] {$v_3$};

  \draw (A) -- (C);
  \draw (C) -- (B);
  \draw (A) -- (B);
  \draw (B) -- (F);
  \draw (B) -- (G);
  \draw (D) -- (A);
  \draw (E) -- (A);
  \fill[black] (A) circle (2pt);
  \fill[black] (B) circle (2pt);
  \fill[black] (C) circle (2pt);
  \fill[black] (D) circle (2pt);
  \fill[black] (E) circle (2pt);
  \fill[black] (F) circle (2pt);
  \fill[black] (G) circle (2pt);

  \draw (-2,0.7) node[anchor=north west] {$a\biggl\{$};  
  \draw (3,0.7) node[anchor=north west] {$\biggl\}n-3-a$};

  \draw[loosely dotted, line width=1.5pt] (3.00, 0.50) -- (3.00, -0.50);
  \draw[loosely dotted, line width=1.5pt] (-1.00, 0.50) -- (-1.00, -0.50);
\end{tikzpicture}
        \caption{$U_1$}
    \end{subfigure}
    \hspace{3cm}
\begin{subfigure}[b]{0.3\textwidth}
        \centering
        \begin{tikzpicture}[scale=0.8]
  \coordinate (A) at (0.00, 0.00);
  \coordinate (B) at (2.00, 0.00);
  \coordinate (C) at (1.00, 2.00);
  \coordinate (D) at (3.00, 0.00);
  \coordinate (E) at (4.00, 0.50);
  \coordinate (F) at (4.00, -0.50);

  \draw (A) -- (C);
  \draw (C) -- (B);
  \draw (A) -- (B);
  \draw (B) -- (D);
  \draw (D) -- (E);
  \draw (D) -- (F);
  \fill[black] (A) circle (2pt);
  \fill[black] (B) circle (2pt);
  \fill[black] (C) circle (2pt);
  \fill[black] (D) circle (2pt);
  \fill[black] (E) circle (2pt);
  \fill[black] (F) circle (2pt);

  \draw (0,0) node[anchor= north] {$v_1$};
  \draw (1,2) node[anchor= south] {$v_2$};
  \draw (2,0) node[anchor= north] {$v_3$};

  \draw (4,0.7) node[anchor=north west] {$\biggl\}n-4$};
  \draw[loosely dotted, line width=1.5pt] (4.00, 0.50) -- (4.00, -0.50);
\end{tikzpicture}
        \caption{$U_2$}
    \end{subfigure}
    \caption{Unicyclic graphs with a $3$-cycle and $\gamma = 2$.}
    \label{mat_2_cycle_3}
\end{figure}

We consider the following two cases:
\begin{itemize}
    \item[Case 1:] Suppose that $G \cong U_1$. Let $G^*$ be the graph obtained from $G$ by subdividing the edge $v_1 v_2$. By Lemma~\ref{lem-sub}, it follows that $\rho(G^*) < \rho(G)$. Note that $G^*$ is a graph on $n+1$ vertices and the length of its unique cycle is $4$. Let $G^{**}$ be the subgraph of $G^*$ obtained by deleting a pendant vertex. Then $G^{**} \in \mathcal{U}_{n,2}$, and by Lemma~\ref{lem:subgraph} we have $\rho(G^{**}) < \rho(G^*) < \rho(G)$. Consequently, 
    \[
    \rho\left(G_1\left( \left\lfloor \frac{n-4}{2} \right\rfloor, \left\lceil \frac{n-4}{2} \right\rceil\right)\right) < \rho(U_1).
    \]

    \item[Case 2:] Suppose that $G \cong U_2$. The characteristic polynomial of $U_2$ is given by
    \[
    \phi_{\lambda}(U_2) = \lambda^{n-5}(\lambda+1) f(\lambda),
    \]
    where $f(\lambda) = \lambda^4-\lambda^3-(n-1)\lambda^2+(n-3)\lambda+2(n-4).$ Since the spectral radius of a graph is the largest root of its characteristic polynomial, $\rho(U_2)$ is the largest root of $f(\lambda)$.
    
    Next, we evaluate the spectral radius of the graph $G_1\left( \left\lfloor \frac{n-4}{2} \right\rfloor, \left\lceil \frac{n-4}{2} \right\rceil\right)$.
    
    \begin{itemize}
        \item[Subcase 2.1] ($n$ is even): Let $n=2k$. Then $G_1\left( \left\lfloor \frac{n-4}{2} \right\rfloor, \left\lceil \frac{n-4}{2} \right\rceil\right) \cong G_1(k-2,k-2)$. The characteristic polynomial of this graph is 
        \[
        \phi_{\lambda}\left(G_1(k-2,k-2)\right) = \lambda^{n-4}(\lambda^2-k+2)(\lambda^2-k-2).
        \]
        Thus, $\rho(G_1(k-2,k-2))= \sqrt{k+2}$. A direct calculation yields \[
    f(\sqrt{k+2}) = -k^2+5k-2+(k-5)\sqrt{k+2} < 0,
    \]
    for all $k \in \mathbb{N}$. Since $f(\lambda)$ is a monic polynomial, its largest root must be strictly greater than $\sqrt{k+2}$. Thus, we have
    \[
    \rho(G_1(k-2,k-2)< \rho(U_2).
    \]
        
        \item[Subcase 2.2] ($n$ is odd): Let $n=2k+1$. Then $G_1\left( \left\lfloor \frac{n-4}{2} \right\rfloor, \left\lceil \frac{n-4}{2} \right\rceil\right) \cong G_1(k-2,k-1)$. Its characteristic polynomial is given by
        \[
        \phi_{\lambda}\left(G_1(k-2,k-1)\right) = \lambda^{n-4}\left(\lambda^4-(2k+1)\lambda^2+(k^2+k-4)\right),
        \]
        which implies that $\rho(G_1(k-2,k-1)) = \sqrt{k+\frac{\sqrt{17}+1}{2}}$.
    \end{itemize}

    In a similar way to Case 3.1, we have
    \begin{align*}
        &f\left(\sqrt{k+\frac{\sqrt{17}+1}{2}}\right) \\&= -k^2+4k+\left(\frac{\sqrt{17}+1}{2}\right)^2-6+\left(k-2-\frac{\sqrt{17}+1}{2}\right) \cdot \sqrt{k+\frac{\sqrt{17}+1}{2}} < 0,
    \end{align*}
    for all $k \in \mathbb{N}$. Thus, $  \rho(G_1(k-2,k-1)< \rho(U_2).$
\end{itemize}

This completes the proof.

\end{proof}

\section{Minimal spectral radius of unicyclic graphs with $\gamma = 3$}
The main objective of this section is to determine the minimum spectral radius among all graphs in $\mathcal{U}_{n,3}$. For our purpose, let us define the following class of unicyclic graphs with matching number $3$.
\bd\label{def:cycle-4-1} Let $G_2$ be a unicyclic graph on $6$ vertices obtained from a path $P_6=[v_1,v_2,\ldots,v_6]$ by adding an edge between the vertices $v_3$ and $v_6$. Let $a$, $b$ and $c$ be nonnegative integers and $G_2(a,b,c)$ be the class of graphs on $n$ vertices obtained from $G_2$ by attaching $a$, $b$, and $c$ pendant vertices at the vertices $v_1$, $v_3$ and $v_5$, respectively, where $a+b+c = n-6$ {\rm(}see Figure~\ref{uni_min_1}{\rm)}. \ed
\begin{figure}[h]
	\centering
	\begin{tikzpicture}[scale=1]
  \coordinate (A) at (0.00, 0.00);
  \coordinate (B) at (2.00, 0.00);
  \coordinate (C) at (0.00, 2.00);
  \coordinate (D) at (2.00, 2.00);
  \coordinate (E) at (1.50, -1.00);
  \coordinate (F) at (2.50, -1.00);
  \coordinate (G) at (-1.00, 1.50);
  \coordinate (H) at (-1.00, 2.50);
  \coordinate (I) at (3.00, 0.00);
  \coordinate (J) at (4.00, 0.00);
  \coordinate (K) at (5.00, 0.50);
  \coordinate (L) at (5.00, -0.50);

  \draw (0,0) node[anchor= east] {$v_6$};
  \draw (0,2) node[anchor= south] {$v_5$};
  \draw (2,0) node[anchor= south west] {$v_3$};
  \draw (2,2) node[anchor= west] {$v_4$};
  \draw (3,0) node[anchor= south] {$v_2$};
  \draw (4,0) node[anchor= south] {$v_1$};

  \draw (2.7,-0.98) node[anchor=north west, rotate=-90] {$\Biggl\}$};
  \draw (1.8,-1.25) node[anchor=north west] {$b$};
  \draw (5,0.7) node[anchor=north west] {$\Biggl\}a$};
  \draw (-1.7,2.7) node[anchor=north west] {$c\Biggl\{$};

  \draw (C) -- (D);
  \draw (C) -- (A);
  \draw (A) -- (B);
  \draw (D) -- (B);
  \draw (B) -- (I);
  \draw (I) -- (J);
  \draw (J) -- (K);
  \draw (J) -- (L);
  \draw (H) -- (C);
  \draw (G) -- (C);
  \draw (B) -- (E);
  \draw (B) -- (F);
  \fill[black] (A) circle (2pt);
  \fill[black] (B) circle (2pt);
  \fill[black] (C) circle (2pt);
  \fill[black] (D) circle (2pt);
  \fill[black] (E) circle (2pt);
  \fill[black] (F) circle (2pt);
  \fill[black] (G) circle (2pt);
  \fill[black] (H) circle (2pt);
  \fill[black] (I) circle (2pt);
  \fill[black] (J) circle (2pt);
  \fill[black] (K) circle (2pt);
  \fill[black] (L) circle (2pt);

  \draw[loosely dotted, line width=1.5pt] (5,0.5) -- (5,-0.5);
  \draw[loosely dotted, line width=1.5pt] (1.50, -1.00) -- (2.50, -1.00);
  \draw[loosely dotted, line width=1.5pt] (-1.00, 1.50) -- (-1.00, 2.50);

\end{tikzpicture}
	\caption{The unicyclic graph $G_2(a,b,c)$.}
	\label{uni_min_1}
\end{figure}
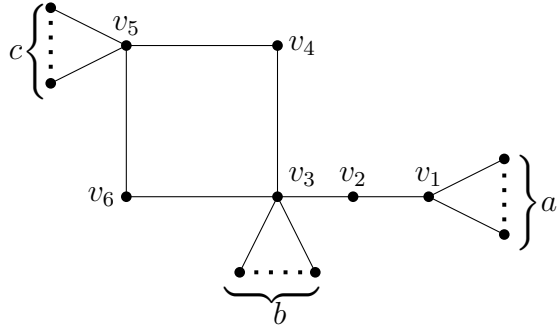

Next, we define a graph that belongs to the class of graphs $G_2(a,b,c)$, which will be used later.
\begin{defn}
	If $n\geq 12$ is an integer such that $n=3k+l$, where $k$ is a positive integer and $l\in \{-1,0,1\}$, then let us denote the graph $G_2(k,k+l-4,k-2)$ by $U_{n,3}^*$.
\end{defn}

We begin this section by proving that if $a$, $b$, and $c$ are nonnegative integers satisfying $a+b+c=n-6$ and $n\geq 12$, then $\rho(G_2(a,b,c))\geq \rho(U_{n,3}^*)$ with equality if and only if $G_2(a,b,c))\cong U_{n,3}^*$. Since the spectral radius of a graph equals the largest root of its characteristic polynomial, we begin by computing $\phi_\lbd(G_2(a,b,c))$, the characteristic polynomial of $G_2(a,b,c)$. Observe that
\begin{align*}
    \phi_\lbd(G_2(a,b,c))=&\lbd^{n-6}\big(\lbd^6-n\lbd^4+(ab+bc+ac+5a+3b+4c+6)\lbd^2\\&-(abc+2ab+3ac+bc+2a+2b+2c)\big).
\end{align*}

Let
$$f(a,b,c)=ab+bc+ac+5a+3b+4c+6, \quad
g(a,b,c)=abc+2ab+3ac+bc+2a+2b+2c,$$
and define
$$\psi_\lbd(a,b,c)=\lbd^3-n\lbd^2+f(a,b,c)\lbd-g(a,b,c).$$
Then $\psi_\lbd(a,b,c)$ is a real-rooted monic polynomial. Let $\mu_1(a,b,c)$ denote the largest root of $\psi_\lbd(a,b,c)$. Since $\rho(G_2(a,b,c))=\sqrt{\mu_1(a,b,c)}$, minimizing $\rho(G_2(a,b,c))$ is equivalent to minimizing $\mu_1(a,b,c)$, where $a+b+c$ is a constant. Note that the spectral radius of the star graph $S_k$ on $k$ vertices is $\rho(S_k)=\sqrt{k-1}$. Since $G_2(a,b,c)$ contains $S_{a+2}$, $S_{b+4}$, and $S_{c+3}$ as induced subgraphs, it follows that
$\mu_1(a,b,c)>a+1$, $\mu_1(a,b,c)>b+3$, and $\mu_1(a,b,c)>c+2.$

We first investigate the behavior of $\mu_1(a,b,c)$ under various distributions of $a$, $b$, and $c$ through the following preliminary results.
\bl\label{qq} If $a$ is a nonnegative integer, then $\mu_1(a,a,a)=a+3+\sqrt{3}$.\el 
\pf Note that $\mu_1(a,a,a)$ is the largest root of $\lbd^3-3(a+2)\lbd^2+3(a^2+4a+2)\lbd-a(a^2+6a+6)$. Observe that $a$ is a root of this polynomial. Thus,
\begin{align*}
&\lbd^3-3(a+2)\lbd^2+3(a^2+4a+2)\lbd-a(a^2+6a+6)\\
&=(\lbd-a)(\lbd^2-2(a+3)\lbd+(a^2+6a+6))\\
&=(\lbd-a)(\lbd-a-3-\sqrt{3})(\lbd-a-3+\sqrt{3}).
\end{align*}
Hence, $\mu_1(a,a,a)=a+3+\sqrt{3}$.
\qe

\bl\label{ab2} If $a\geq 4$ is an integer, then $\mu_1(a,a-2,a-1)>\mu_1(a+1,a-3,a-1)$. \el 
\pf Direct calculation shows that 
\begin{align*}
\psi_\lbd(a,a-2,a-1)&=\lbd^3-3(a+1)\lbd^2+(3a^2+6a-2)\lbd-(a^3+3a^2-2a-4)\\
&=(\lbd-a-1)(\lbd^2-2(a+1)\lbd+(a^2+2a-4))\\
&=(\lbd-a-1)(\lbd-a-1-\sqrt{5})(\lbd-a-1+\sqrt{5}).
\end{align*}
The roots of the polynomial $\psi_\lbd(a,a-2,a-1)$ are $a+1$, $a+1+\sqrt{5}$, and $a+1-\sqrt{5}$. Hence, $\mu_1(a,a-2,a-1)=a+1+\sqrt{5}$.

Similarly, we have
\begin{align*}
	\psi_\lbd(a+1,a-3,a-1)&=\lbd^3-3(a+1)\lbd^2+3(a^2+2a-1)\lbd-(a^3+3a^2-3a-9)\\
	&=(\lbd-a-3)(\lbd^2-2a\lbd+(a^2-3))\\
	&=(\lbd-a-3)(\lbd-a-\sqrt{3})(\lbd-a+\sqrt{3}).
\end{align*}
Thus, the largest root of the polynomial $\psi_\lbd(a+1,a-3,a-1)$ is $\mu_1(a+1,a-3,a-1)=a+3$, which implies the desired inequality.
\qe

\bl\label{ab3} If $a\geq 5$ is an integer, then $\mu_1(a,a-3,a-1)>\mu_1(a+1,a-4,a-1)$. \el 
\pf It can be verified that 
\begin{align*}
\psi_\lbd(a,a-3,a-1)&=\lbd^3-(3a+2)\lbd^2+(3a^2+4a-4)\lbd-(a^3+2a^2-4a-5)\\
&=(\lbd-a-1)\left(\lbd-a-\frac{1+\sqrt{21}}{2}\right)\left(\lbd-a-\frac{1-\sqrt{21}}{2}\right).
\end{align*}
It follows that $\mu_1(a,a-3,a-1)=a+\frac{1+\sqrt{21}}{2}$.

Again, 
\begin{align*}
\psi_\lbd(a+1,a-4,a-1)&=\lbd^3-(3a+2)\lbd^2+(3a^2+4a-6)\lbd-(a^3+2a^2-6a-11)\\
&=(\lbd-a)^3-2(\lbd-a)^2-6(\lbd-a)+11.
\end{align*}
Observe that $\dfrac{d^2}{d\lbd^2}\psi_\lbd(a+1,a-4,a-1)=6(\lbd-a)-4>0$ for $\lbd\geq a+\frac{1+\sqrt{21}}{2}$. Thus, $$\dfrac{d}{d\lbd}\psi_\lbd(a,a-3,a-1)\geq 3\left(\frac{1+\sqrt{21}}{2}\right)^2-4\left(\frac{1+\sqrt{21}}{2}\right)-6=\frac{17-\sqrt{21}}{2}>0$$ for $\lbd\geq a+\frac{1+\sqrt{21}}{2}$. It follows that $$\psi_\lbd(a,a-3,a-1)\geq \left(\frac{1+\sqrt{21}}{2}\right)^3-2\left(\frac{1+\sqrt{21}}{2}\right)^2-6\left(\frac{1+\sqrt{21}}{2}\right)+11=5-\sqrt{21}>0$$ for $\lbd\geq a+\frac{1+\sqrt{21}}{2}$. Since $\psi_\lbd(a+1,a-4,a-1)=5-\sqrt{21}>0$ when $\lbd=\mu_1(a,a-3,a-1)$, we have $\mu_1(a+1,a-4,a-1)<\mu_1(a,a-3,a-1)$.\qe

\bl\label{ab4ac3} If $a\geq 5$ is an integer, then $\mu_1(a,a-4,a-3)>\mu_1(a,a-5,a-2)$. \el 
\pf Observe that  $\psi_\lbd(a,a-4,a-3)-\psi_\lbd(a,a-5,a-2)=\lbd-a-2$ and $\psi_{a+2}(a,a-5,a-2)=2$. Again, $\frac{d}{d\lbd}\psi_\lbd(a,a-5,a-2)=3(\lbd-a)^2+2(\lbd-a)-7\geq 9$ for $\lbd\geq a+2$. Thus, $\mu_1(a,a-5,a-2)<a+2$. Hence, the result follows from Lemma \ref{poly}(ii).
\qe

\bl\label{ab4} If $a$ and $b$ are nonnegative integers such that $a-b\geq 4$, then $\mu_1(a,b,a-1)>\mu_1(a,b+1,a-2)$. \el 
\pf It is easy to see that \begin{align*}
    \psi_\lbd(a,b,a-1)=&\lbd^3-(2a+b+5)\lbd^2+(a^2+2ab+8a+2b+2)\lbd-(a^2b+3a^2+2ab+a+b-2)\\
    =&\left(\lbd-a-1\right)\left(\lbd^2-(a+b+4)\lbd+(ab+3a+b-2)\right)
\end{align*} and $$\psi_\lbd(a,b+1,a-2)-\psi_\lbd(a,b,a-1)=(a-b-3)\left(\lbd-\left(a+1+\frac{1}{a-b-3}\right)\right).$$
The largest root of the polynomial $\psi_\lbd(a,b,a-1)$ is  given by $$\mu_1(a,a-4,a-3)=\frac{a+b+4+\sqrt{(a-b-2)^2+20}}{2}.$$ We aim to show that $ \mu_1(a,a-4,a-3)>a+1+\frac{1}{a-b-3}$, which reduces to the inequality $$ \sqrt{(a-b-2)^2+20}-(a-b-2)-\frac{2}{a-b-3}>0.$$

Let $t=a-b-2$. Then $t\geq 2$. Thus, we have $\frac{1}{t-1}+\frac{1}{(t-1)^2}<4$. A straightforward calculation implies that $$\sqrt{t^2+20}>t+\frac{2}{t-1}.$$

Thus, we have $\psi_\lbd(a,b+1,a-2)>\psi_\lbd(a,b,a-1)$ for $\lbd\geq \mu_1(a,b,a-1)$. By Lemma \ref{poly}(i), we conclude $\mu_1(a,b,a-1)>\mu_1(a,b+1,a-2)$.
\qe

\bl\label{ab5ac3}  Let $a$, $b$, and $c$ be nonnegative integers such that $a-b\geq 5$ and $a-c\geq 3$. If $a-c$ is odd, then $\mu_1(a,b,c)>\mu_1(a-1,b+1,c)$. \el 
\pf Since $a-c\geq 3$, we have $\mu_1(a,b,c)>a+1\geq c+4$. Further, we have $$\psi_\lbd(a-1,b+1,c)-\psi_\lbd(a,b,c)=(a-b-3)\left(\lbd-\left(c+2+\frac{4}{a-b-3}\right)\right).$$
Since $c+2+\frac{4}{a-b-3}\leq c+4$, we have $\psi_\lbd(a-1,b+1,c)>\psi_\lbd(a,b,c)$ for $\lbd\geq \mu_1(a,b,c)$. Thus, Lemma \ref{poly}(i) follows the required inequality.
\qe

\bl\label{a00}  If $a\geq 6$ is an integer, then $\mu_1(a,0,0)>\mu_1(a-1,0,1)$. \el 
\pf Since $\psi_\lbd(a-1,0,1)-\psi_\lbd(a,0,0)=(a-2)\left(\lbd-\left(3+\frac{3}{a-2}\right)\right)$ and $3+\frac{3}{a-2}<\mu_1(a,0,0)$, the result follows from Lemma \ref{poly}(i).
\qe

\bl\label{0b0}  If $b\geq 6$ is an integer, then $\mu_1(0,b,0)>\mu_1(0,b-1
,1)$. \el 
\pf Since $\psi_\lbd(0,b-1,1)-\psi_\lbd(0,b,0)=b\left(\lbd-\left(1-\frac{1}{b}\right)\right)$ and $1-\frac{1}{b}<\mu_1(0,b,0)$, the result follows from Lemma \ref{poly}(i).\qe

%
%

We are now ready to determine the minimum value of $\mu_1(a,b,c)$ subject to $a+b+c=n-6$, where $n\geq 12$ is fixed. We begin by characterizing the structure of the minimizing parameters.
\bl\label{main2} Let $n\geq 6$ be a fixed integer and
$$\mu_1(p,q,r)=\min\{\mu_1(a,b,c)\mid a,b,c\ge 0 \text{ are integers and } a+b+c=n\}.$$
Then $r=p-2$ and $p-q\geq 3$. \el 
\pf We first show that $r\geq q+1$. If $q=0$, then Lemma \ref{a00} implies that $r\geq 1$. On the contrary, assume that $r\leq q$ and $q\geq 1$. A direct computation gives
$$\psi_\lbd(p,q-1,r+1)-\psi_\lbd(p,q,r)=(q-r)\lbd-((q-r)(p+1)-1).$$
Since $p+1<\mu_1(p,q,r)$, it follows that $\psi_\lbd(p,q-1,r+1)>\psi_\lbd(p,q,r)$ for $\lbd\geq \mu_1(p,q,r)$. By Lemma~\ref{poly}(i), we obtain $\mu_1(p,q-1,r+1)<\mu_1(p,q,r)$, which contradicts the minimality of $\mu_1(p,q,r)$. Hence, $r\geq q+1$. 

Next, we prove that $p\geq r+1$. If $r=0$, them Lemma \ref{0b0} implies that $p\geq 1$. On the contrary, assume that $r\geq p$ and $r\geq 1$. A similar calculation as above implies that $$\psi_\lbd(p+1,q,r-1)>\psi_\lbd(p,q,r)\text{ for }\lbd\geq \mu_1(p,q,r).$$ Consequently, we have $\mu_1(p+1,q,r-1)<\mu_1(p,q,r)$ (by Lemma \ref{poly}(i)). It follows that $p\geq r+1$. 

We now consider the following two cases:

\bdsc
\item{\rm Case 1.} Suppose that $p-r$ is  an even positive integer. If $p-r=2$, then we are done. Let $p-r=2+2t$, where $t\ge 1$ is an integer. Then $$\psi_\lbd(p-1,q,r+1)-\psi_\lbd(p,q,r)=2t\left(\lbd-\left(q+3+\frac{3}{2t}\right)\right).$$ Since $t\geq 1$, Lemma \ref{qq} implies that $q+3+\frac{3}{2t}<q+3+\sqrt{3}=\mu_1(q,q,q)<\mu_1(p,q,r)$. Thus, $\psi_\lbd(p-1,q,r+1)>\psi_\lbd(p,q,r)$ for $\lbd\geq\mu_1(p,q,r)$. By Lemma \ref{poly}(i), it follows that $\mu_1(p-1,q,r+1)<\mu_1(p,q,r)$. By successively applying the above argument, we have
$$\mu_1(p,q,r)>\mu_1(p-1,q,r+1)>\cdots>\mu_1(p-t,q,r+t)=\mu_1(p_1,q_1,r_1),$$ where $r_1=p_1-2$ and $p_1-q_1=p-q-t\geq p-r-t+1=3+t\geq 4$.

\item{\rm Case 2.} Suppose that $p-r$ is an odd positive integer. First, assume $p-r=1$. Then $p-q \geq 2$. If $p-q=2$, then by Lemma \ref{ab2}, $\mu_1(p,p-2,p-1)>\mu_1(p+1,p-3,p-1).$ If $p-q=3$, then Lemma \ref{ab3} follows that $\mu_1(p,p-3,p-1)>\mu_1(p+1,p-4,p-1).$ If $p-q\geq 4$, then from Lemma \ref{ab4}, $\mu_1(p,q,p-1)>\mu_1(p,q+1,p-2)=\mu_1(p_1,q_1,r_1)$, where $p_1=r_1+2$ and $p_1-q_1\geq 3$

Now, assume $p-r=1+2t$, where $t\geq 1$ is an integer. Note that $p-q\geq 4$. If $p-q=4$, then $p-r=3$ and
$$\mu_1(p,q,r)=\mu_1(p,p-4,p-3)>\mu_1(p,p-5,p-2) \quad\text{(by Lemma \ref{ab4ac3})}$$
If $p-q\geq 5$, then Lemma \ref{ab5ac3} implies that $\mu_1(p,q,p-1-2t)>\mu_1(p-1,q+1,p-1-2t)$. If $t=1$, we are done. Otherwise, repeating the arguments from Case 1, we obtain 
$$\mu_1(p-1,q+1,p-1-2t)>\mu_1(p-t,q+1,p-2-t)=\mu_1(p_1,q_1,r_1),$$ where $p_1=r_1+2$ and $p_1-q_1=p-q-t-1\geq p-r-t=t+1\geq 3$ (since $p-r=2t+1\geq 5$).
\edsc
This completes the proof.\qe

Next, we minimize $\mu_1(a,b,c)$ with the condition $c=a-2$. For fixed $n$, $\psi_\lbd(a,n-2a-4,a-2)$ can be written as $\psi_\lbd(a)$ and $\mu_1(a,n-2a-4,a-2)$ can be written as $\mu_1(a)$. In the following result, we determine the minimum value of $\mu_1(a)$.
\bl\label{main1} Let $a$ and $n$ be integers such that $n\geq 12$ and $2\leq a\leq \frac{n-4}{2}$. Suppose that $\mu_1(a)$ is the largest root of the real-rooted polynomial 
$$\psi_\lbd(a)=\lbd^3-n\lbd^2+(-3a^2+(2n-3)a+n-6)\lbd-(-2a^3+(n-3)a^2+(n-6)a-4).$$
If $n=3k+l$ for some integers $k\geq 4$ and $l\in \{-1,0,1\}$, then $\mu_1(a)\geq\mu_1(k)$ with equality if and only if $a=k$.
\el
\pf Let $\psi_\lbd(a)=\lbd^3-n\lbd^2+f(a)\lbd-g(a)$. Then $f(a)=-3a^2+(2n-3)a+n-6$ and $g(a)=-2a^3+(n-3)a^2+(n-6)a-4$.

It is easy to verify that $$f(a)-f(a-1)=2n-6a$$
and $$g(a)-g(a-1)=a(2n-6a)-5.$$
For $a\geq 3$, \be\label{eq1}\psi_\lbd(a)-\psi_\lbd(a-1)=(2n-6a)\lbd-(a(2n-6a)-5).\ee
Observe that $\mu_1(a-1)>a$. We now have the following three cases:

\bdsc
\item{Case 1.} Let $n=3k-1$ for some integer $k\geq 4$. Then \eqref{eq1} becomes $$\psi_\lbd(a)-\psi_\lbd(a-1)=(6(k-a)-2)\lbd-(a(6(k-a)-2)-5).$$

If $k\geq a+1$, then $\psi_\lbd(a)>\psi_\lbd(a-1)$ for $\lbd\geq \mu_1(a-1)$. From Lemma \ref{poly}(i), it follows that $\mu_1(a)<\mu_1(a-1)$. Hence, $\mu_1(k-1)<\mu_1(k-1-r)$ for every integer $1\leq r\leq k-3$.

If $k\leq a-1$, then $a+\frac{5}{6(a-k)+2}<a+1<\mu_1(a)$. Thus, we have $\psi_\lbd(a-1)>\psi_\lbd(a)$ for $\lbd\geq\mu_1(a)$. From Lemma \ref{poly}(i), it follows that $\mu_1(a-1)<\mu(a)$. Hence, $\mu_1(k)<\mu_1(k+r)$ for every integer $1\leq r\leq \lfloor\frac{n-4}{2}\rfloor-k$.

If $k=a$, then $\psi_\lbd(a)=(\lbd-a)^3+(\lbd-a)^2-7(\lbd-a)+4$. A straightforward computation shows that
$$\psi_{a+2}(a)=2>0$$ and $$\frac{d}{d\lbd}\psi_{\lbd}(a)=3(\lbd-a)^2+2(\lbd-a)-7\geq 9>0\text{ for } \lbd\geq a+2.$$
Thus, $\mu_1(a)<a+2$. Since $\psi_\lbd(a-1)-\psi_\lbd(a)=2\left(\lbd-\left(a+\dfrac{5}{2}\right)\right)$, Lemma \ref{poly}(ii) implies that $\mu_1(a)<\mu_1(a-1)$, i.e., $\mu_1(k)<\mu_1(k-1)$. It follows that $\mu_1(a)\geq \mu_1(k)$ with equality if and only if $a=k$. 

\item {Case 2.} Let $n=3k$ for some positive integer $k\geq 4$. Then Equation \eqref{eq1} reduces to $$\psi_\lbd(a)-\psi_\lbd(a-1)=6(k-a)\lbd-(6a(k-a)-5).$$

Note that $a-\frac{5}{6(k-a)}<a<\mu_1(a-1)$ for $k>a$. Thus, if $a\leq k$, then $\psi_\lbd(a)>\psi_\lbd(a-1)$ for $\lbd\geq \mu_1(a-1)$. Consequently, we have $\mu_1(a)<\mu_1(a-1)$ when $k\geq a$. Hence, $\mu_1(k)<\mu_1(k-r)$ for every integer $1\leq r\leq k-2$.

Now, if $a\geq k+1$, then we have $\frac{5}{6(a-k)}<1$ and$$\psi_\lbd(a)-\psi_\lbd(a-1)=6(k-a)\left(\lbd-\left(a+\frac{5}{6(a-k)}\right)\right).$$ 
Thus, for $\lbd\geq \mu_1(a)$, we have $\psi_\lbd(a-1)>\psi_\lbd(a)$. It follows that $\mu_1(a-1)<\mu_1(a)$ for $a\geq k+1$. Hence, $\mu_1(k)<\mu_1(k+r)$ for every integer $1\leq r\leq \lfloor\frac{n-4}{2}\rfloor-k$. Combining we get $\mu_1(a)\geq \mu_1(k)$ with equality if and only if $a=k$.

\item {Case 3.} Let $n=3k+1$ for some positive integer $k$. Then Equation \eqref{eq1} reduces to $$\psi_\lbd(a)-\psi_\lbd(a-1)=(6(k-a)+2)\left(\lbd-\left(a-\frac{5}{6(k-a)+2}\right)\right).$$

If $a\leq k$, then $a-\frac{5}{6(k-a)+2}<a<\mu_1(a-1)$. Thus, $\psi_\lbd(a)>\psi_\lbd(a-1)$ for $\lbd\geq\mu_1(a-1)$. Consequently, we have $\mu_1(a)<\mu_1(a-1)$. Hence, $\mu_1(k)<\mu_1(k-r)$ for every integer $1\leq r\leq k-2$.

Now, if $a\geq k+2$, then $a+\frac{5}{6(a-k)-2}<a+1$. Notice that $$\psi_\lbd(a)-\psi_\lbd(a-1)=(6(k-a)+2)\left(\lbd-\left(a+\frac{5}{6(a-k)-2}\right)\right).$$
Thus, for $\lbd\geq\mu_1(a)$, we have $\psi_\lbd(a-1)>\psi_\lbd(a)$. It follows that $\mu_1(a-1)<\mu_1(a)$. Hence, $\mu_1(k+1)<\mu_1(k+1+r)$ for every integer $1\leq r\leq \lfloor\frac{n-4}{2}\rfloor-k-1$. 

Finally if $a=k+1$, then $\psi_\lbd(a-1)-\psi_\lbd(a)=4(\lbd-(a+\frac{5}{4}))$. One can verify that $\psi_{a+\frac{5}{4}}(a)<0$. Thus, $a+\frac{5}{4}<\mu_1(a)$. It follows that $\psi_\lbd(a-1)>\psi_\lbd(a)$ for $\lbd\geq\mu_1(a)$. Therefore, $\mu_1(k)<\mu_1(k+1)$ (by Lemma \ref{poly}(i)). Combining all of them, we get $\mu_1(a)\geq \mu_1(k)$ with equality if and only if $a=k$.
\edsc
This completes the proof. \qe

\bl \label{cycle_4_bound}
Let $n\geq 14$ be an integer such that $n=3a+l$, where $k\geq 4$ and $l\in \{-1,0,1\}$. Then 
\bdsc
\item {\rm (i)} $\mu_1(a,a-5,a-2)< a+2$ if $l=-1$,
\item {\rm (ii)} $\mu_1(a,a-4,a-2)= a+2$ if $l=0$, and
\item {\rm (iii)} $\mu_1(a,a-3,a-2)< a+3$ if $l=1$.
\edsc
\el
\pf Notice that $\mu_1(a,a-4+l,a-2)$ is the largest root of the polynomial $$f(\lambda)=\lbd^3-3a\lbd^2+(3a^2-6)\lbd-(a^3-6a-4)-l(\lbd^2-(2a+1)\lbd+a(a+1)).$$ 
A direct calculation yields $f(a+2)=-2l$, $f(a+3)=13-6l$, $f'''(a+2)=6$, $f''(a+2)=12-2l$ and $f'(a+2)=6-3l$. Thus, $f$ is a strictly increasing function for $x\geq a+2$. Hence, the required result follows.\qe

In the next result, we find the minimum possible value of $\mu_1(a,b,c)$ and for which values of $a$, $b$, and $c$ it is attained.
\bt\label{main3} Let $n\geq 12$ be an integer. Suppose that $a$, $b$, and $c$ are three nonnegative integers such that $n=a+b+c+6$. If $n=3k+l$ for some integers $k\geq 4$ and $l\in \{-1,0,1\}$, then $\mu_1(a,b,c)\geq \mu_1(k,k+l-4,k-2)$ with equality if and only if $a=k$, $b=k+l-4$ and $c=k-2$.
\et
\pf Let $\mu_1(p,q,r)=\min\{\mu_1(a,b,c)\mid a,b,c\in \mathbb{Z}_{\geq 0}, \text{ and }a+b+c=n-6\}.$  By Lemma \ref{main2}, $p=r+2$ and $p-q\geq 3$. Now, by Lemma \ref{main1}, we get $p=k$. Thus, $q=n-4-2p=3k+l-4-2k=k+l-4$. This completes the proof.\qe

Since the square root of the largest root of $\psi_\lbd(a,b,c)$ is $\rho(G_2(a,b,c))$, the following is an immediate consequence of Theorem \ref{main3}.

\begin{theorem}\label{thm:min-4-cycle}
	Let $a$, $b$, and $c$ be three nonnegative integers such that $a+b+c+6=n\geq 12$. Then $\rho(G_2(a,b,c) \geq \rho(U_{n,3}^*)$ with equality if and only if $G_2(a,b,c)\cong U_{n,3}^*$.
\end{theorem}
\pf The proof follows from Theorem \ref{main3}.\qe

By Theorem~\ref{thm:min-4-cycle}, $U_{n,3}^*$ minimizes the spectral radius over the family $\{G_2(a,b,c): a+b+c+6=n\}$. It remains to compare $\rho(U_{n,3}^*)$ with the spectral radii of the remaining graphs in $\mathcal{U}_{n,3}$.
\subsection{Unicyclic graphs with girth $6$ and $\gamma=3$}
In this subsection, we compare the spectral radius of $U_{n,3}^*$ with the spectral radii of graphs in $\mathcal U_{n,3}$ whose unique cycle is a $6$-cycle. Let $a,b,c$ be three nonnegative integers, and at least one of them is nonzero. Without loss of generality, assume that $a> 0$. Let $G_3 = [v_1,v_2,v_3,v_4,v_5,v_6,v_1]$ be the cycle on $6$ vertices. We define $G_3(a,b,c)$ to be the graph obtained from the cycle $G_3$ by attaching $a$, $b$, and $c$ pendant vertices to $v_1$, $v_3$, and $v_5$, respectively (see Figure \ref{fig:6-2}).
\begin{figure}[ht]
		\centering
		\begin{tikzpicture}[scale=1]
			\coordinate (A) at (0.00, 0.00);
			\coordinate (B) at (1.00, 1.00);
			\coordinate (C) at (-1.00, 1.00);
			\coordinate (D) at (-1.00, 2.00);
			\coordinate (E) at (1.00, 2.00);
			\coordinate (F) at (0.00, 3.00);
			\coordinate (G) at (-0.50, 4);
			\coordinate (H) at (0.50, 4);
			\coordinate (I) at (1.75, 0.50);
			\coordinate (J) at (1.75, 1.50);
			\coordinate (K) at (-1.75, 1.50);
			\coordinate (L) at (-1.75, 0.50);
			
			\draw (-1,1) node[anchor=north] {$v_1$};
			\draw (-1,2) node[anchor=east] {$v_2$};
			\draw (0,3) node[anchor=west] {$v_3$};
			\draw (1,1) node[anchor=north] {$v_5$};
			\draw (1,2) node[anchor=west] {$v_4$};
			\draw (0,0) node[anchor=north] {$v_6$};
			
			\draw[loosely dotted, line width=1.5pt] (-0.50, 4) -- (0.50, 4);
			\draw[loosely dotted, line width=1.5pt] (1.75, 0.50) -- (1.75, 1.50);
			\draw[loosely dotted, line width=1.5pt] (-1.75, 0.50) -- (-1.75, 1.50);
			
			\draw (D) -- (F);
			\draw (D) -- (C);
			\draw (C) -- (A);
			\draw (F) -- (E);
			\draw (E) -- (B);
			\draw (B) -- (A);
			\draw (F) -- (G);
			\draw (F) -- (H);
			\draw (K) -- (C);
			\draw (L) -- (C);
			\draw (B) -- (J);
			\draw (B) -- (I);
			\fill[black] (A) circle (2pt);
			\fill[black] (B) circle (2pt);
			\fill[black] (C) circle (2pt);
			\fill[black] (D) circle (2pt);
			\fill[black] (E) circle (2pt);
			\fill[black] (F) circle (2pt);
			\fill[black] (G) circle (2pt);
			\fill[black] (H) circle (2pt);
			\fill[black] (I) circle (2pt);
			\fill[black] (J) circle (2pt);
			\fill[black] (K) circle (2pt);
			\fill[black] (L) circle (2pt);

            \draw (1.7,1.8) node[anchor=north west, rotate=0] {$\Biggl\}$};
            \draw (2.1,1) node[anchor=west] {$c$};

            \draw (-1.7,0.2) node[anchor=north west, rotate=180] {$\Biggl\}$};
            \draw (-2.1,1) node[anchor=east] {$a$};

            \draw (-0.8,3.9) node[anchor=north west, rotate=90] {$\Biggl\}$};
            \draw (0.3,4.5) node[anchor=east] {$b$};
		\end{tikzpicture}
		\caption{The unicyclic graph $G_3(a,b,c)$.}
		\label{fig:6-2}
	\end{figure}
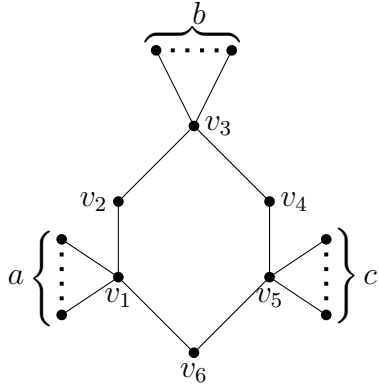
    
\begin{lem}\label{lem-6-cycle}
Let $n\geq 12$ be an integer and write $n=3k+l$, where $k$ is a positive integer and $l\in\{-1,0,1\}$. Suppose that $G\in\mathcal{U}_{n,3}$ and it contains a $6$-cycle. Then
$\rho(G)\geq \rho\bigl(G_3(k+l-2,k-2,k-2)\bigr)$ with equality if and only if $G\cong G_3(k+l-2,k-2,k-2)$.
\end{lem}

\begin{proof}
	Suppose $G \in \mathcal{U}_{n,3}$ and it has a cycle of length $6$. Let $C=[v_1,v_2,v_3,v_4,v_5,v_6,v_1]$ be the unique cycle of $G$. Then we have the following observations:
	\begin{itemize}
		\item[(a)] There are only pendant vertices attached to the vertices of the cycle, otherwise $\gamma > 3$.
		\item[(b)] The pendant vertices cannot be attached to two adjacent vertices of the cycle. Otherwise, assume that $u$ and $v$ are pendant vertices such that $u \sim v_3$ and $v \sim v_4$ (see Figure~\ref{fig:6-1-1}). Then $\Gamma = \{v_1v_2, uv_3, vv_4, v_5v_6\}$ forms a matching of $G$. Thus, $\gamma>3$.
		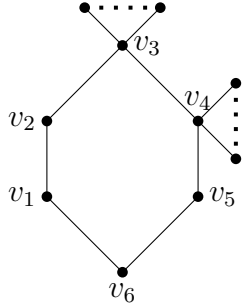
\begin{figure}[ht]
			\centering
			\begin{tikzpicture}[scale=1]
				\coordinate (A) at (0.00, 0.00);
				\coordinate (B) at (1.00, 1.00);
				\coordinate (C) at (-1.00, 1.00);
				\coordinate (D) at (-1.00, 2.00);
				\coordinate (E) at (1.00, 2.00);
				\coordinate (F) at (0.00, 3.00);
				\coordinate (G) at (-0.50, 3.50);
				\coordinate (H) at (0.50, 3.50);
				\coordinate (I) at (1.50, 2.50);
				\coordinate (J) at (1.50, 1.50);
				
				\draw (-1,1) node[anchor= east] {$v_1$};
				\draw (-1,2) node[anchor= east] {$v_2$};
				\draw (0,3) node[anchor=west] {$v_3$};
				\draw (1,1) node[anchor= west] {$v_5$};
				\draw (1,2) node[anchor= south] {$v_4$};
				\draw (0,0) node[anchor= north] {$v_6$};
				
				\draw[loosely dotted, line width=1.5pt] (1.50, 2.50) -- (1.50, 1.50);
				\draw[loosely dotted, line width=1.5pt] (-0.50, 3.50) -- (0.50, 3.50);
				
				\draw (D) -- (F);
				\draw (D) -- (C);
				\draw (C) -- (A);
				\draw (F) -- (E);
				\draw (E) -- (B);
				\draw (B) -- (A);
				\draw (F) -- (G);
				\draw (F) -- (H);
				\draw (E) -- (I);
				\draw (E) -- (J);
				\fill[black] (A) circle (2pt);
				\fill[black] (B) circle (2pt);
				\fill[black] (C) circle (2pt);
				\fill[black] (D) circle (2pt);
				\fill[black] (E) circle (2pt);
				\fill[black] (F) circle (2pt);
				\fill[black] (G) circle (2pt);
				\fill[black] (H) circle (2pt);
				\fill[black] (I) circle (2pt);
				\fill[black] (J) circle (2pt);
			\end{tikzpicture}
			\caption{Unicyclic graph with a $6$-cycle and $\gamma > 3$.}
			\label{fig:6-1-1}
		\end{figure}
		\item[(c)] The pendant vertices have to be distributed among three alternate vertices of the cycle, i.e, distributed among $\{v_1,v_3,v_5\}$ or $\{v_2,v_4,v_6\}$.
	\end{itemize}
	Using the above observations, we can conclude that if $G \in \mathcal{U}_{n,3}$, then $G\cong G_3(a,b,c)$, where $a,b,c$ are nonnegative integers (not all are zero) with $a+b+c+6=n$.
	Finally, using Lemma~\ref{lem-pendant}, we have that the graph with minimum spectral radius among $G_3(a,b,c)$ is the one in which the pendant vertices are distributed almost equally among the vertices $ v_1$, $ v_3$, and $v_5$. Thus, if $n=3k+l$ such that $k$ is a positive integer and $l\in \{-1,0,1\}$, then $\rho(G)\geq \rho(G_3(k+l-2,k-2,k-2))$ with equality if and only if $G\cong G_3(k+l-2,k-2,k-2)$.
\end{proof}

By Lemma~\ref{lem-6-cycle}, the minimum spectral radius among all graphs in $\mathcal U_{n,3}$ whose unique cycle is a $6$-cycle has been determined. We now compare it with $\rho(U_{n,3}^*)$.
\begin{lem}\label{lem:6-4-comp}
	Let $n\geq 12$ be an integer such that $n=3k+l$, where $k$ is a positive integer and $l\in \{-1,0,1\}$. Then $\rho(G_2(k,k+l-4,k-2))\leq \rho(G_3(k+l-2,k-2,k-2))$ with equality if and only if $l=0$.
\end{lem}

\pf Observe that the squares of the largest eigenvalues of the graphs $G_2(k,k+l-4,k-2)$ and $G_3(k+l-2,k-2,k-2)$ are the largest roots of the polynomials
$$\psi_1(\lbd)=\lbd^3-(3k+l)\lbd^2+(3k^2+2lk+l-6)\lbd-(k^3+k^2l+k(l-6)-4)$$ and 
\begin{align*}
	\psi_2(\lbd)=&\lbd^3-(3k+l)\lbd^2+(3k^2+2kl-3)\lbd-(k^3+k^2l-3k-l+2),
\end{align*}
respectively. Thus, we have
\begin{align*}
	\psi_2(\lbd)-\psi_1(\lbd)=(3-l)\left(\lbd-\left(k+1+\frac{3}{3-l}\right)\right)
\end{align*}
If $l=0$, then we can deduce that the largest root of the polynomials $\psi_1(\lbd)$ and $\psi_2(\lbd)$ is the same and its value is $k+2$. Again, one can calculate that the value of $\psi_1(k+1+\frac{3}{3-l})$ is $\frac{7}{8}$ when $l=1$ and $\frac{11}{64}$ when $l=-1$. Thus, the largest root of the polynomial $\psi_1(\lbd)$ is less than $k+1+\frac{3}{3-l}$. Thus, using Lemma \ref{poly}(ii), we get the required inequality.
\qe

The following is the main result of this subsection, which is an immediate consequence of Lemmas~\ref{lem-6-cycle} and \ref{lem:6-4-comp}.
\bt\label{thm-6-cycle}
	Let $n\geq 12$ be an integer and $G \in \mathcal{U}_{n,3}$. If the length of the unique cycle of $G$ is $6$, then $\rho(G)\geq\rho(U_{n,3}^*)$, with equality if and only if $n\equiv 0\pmod 3$ and $G\cong G_3(\frac{n-6}{3},\frac{n-6}{3},\frac{n-6}{3})$.
\et
\pf The proof is immediate.\qe

\subsection{Unicyclic graphs with girth $5$ and $\gamma=3$}
In this subsection, we compare the spectral radius of $U_{n,3}^*$ with the spectral radii of graphs in $\mathcal U_{n,3}$ whose unique cycle is a $5$-cycle. The following result is useful to prove the main result of this subsection. 

\bl\label{lem:5-4-comp} 
Let $H$ be a graph obtained from the path on $6$ vertices $[v_1,v_2,v_3,v_4,v_5,v_6]$ by adding an edge between the vertices $v_1$ and $v_5$. Let $n\geq 12$ be an integer and $U_3$ be the graph on $n$ vertices obtained from the graph $H$ by attaching $n-6$ pendant vertices to $v_6$ {\rm(}see Figure {\rm\ref{fig_5cycle}(a))}. Then $\rho(U_3)>\rho(U_{n,3}^*)$. 
\begin{figure}[ht]
	\centering
	\begin{subfigure}[b]{0.3\textwidth}
		\centering
		\begin{tikzpicture}[scale=1]
			\coordinate (A) at (0.00, 0.00);
			\coordinate (B) at (1.50, 0.00);
			\coordinate (C) at (0.00, 1.50);
			\coordinate (D) at (1.50, 1.50);
			\coordinate (E) at (0.76, 2.77);
			\coordinate (F) at (2.50, 0.00);
			\coordinate (G) at (3.50, -0.50);
			\coordinate (H) at (3.50, 0.50);
			
			\draw (0.76, 2.77) node[anchor= south] {$v_2$};
			\draw (1.5,1.5) node[anchor= west] {$v_1$};
			\draw (0,1.5) node[anchor= east] {$v_3$};
			\draw (1.5, 0.00) node[anchor= north] {$v_5$};
			\draw (0,0) node[anchor= north] {$v_4$};
			\draw (2.5,0) node[anchor=north] {$v_6$};
			\draw (3.5,0.5) node[anchor=west] {$v_7$};
			\draw (3.5,-0.5) node[anchor=west] {$v_n$};

			
			\draw (C) -- (E);
			\draw (C) -- (A);
			\draw (A) -- (B);
			\draw (E) -- (D);
			\draw (D) -- (B);
			\draw (B) -- (F);
			\draw (F) -- (H);
			\draw (F) -- (G);
			\fill[black] (A) circle (2pt);
			\fill[black] (B) circle (2pt);
			\fill[black] (C) circle (2pt);
			\fill[black] (D) circle (2pt);
			\fill[black] (E) circle (2pt);
			\fill[black] (F) circle (2pt);
			\fill[black] (G) circle (2pt);
			\fill[black] (H) circle (2pt);
			
			\draw[loosely dotted, line width=1.5pt] (3.50, -0.50) -- (3.50, 0.50);
		\end{tikzpicture}
		\caption{$U_3$}
	\end{subfigure}
	\hfill
	\begin{subfigure}[b]{0.5\textwidth}
		\centering
		\begin{tikzpicture}[scale=1]
			\coordinate (A) at (0.00, 0.00);
			\coordinate (B) at (2.00, 0.00);
			\coordinate (C) at (0.00, 2.00);
			\coordinate (D) at (2.00, 2.00);
			\coordinate (G) at (3.00, 0.00);
			\coordinate (J) at (4.00, 0.00);
			\coordinate (K) at (5.00, 0.50);
			\coordinate (L) at (5.00, -0.50);
			
			\draw (0,0) node[anchor=north] {$v_3$};
			\draw (2,0) node[anchor=north] {$v_4$};
			\draw (2,2) node[anchor= south] {$v_1$};
			\draw (0,2) node[anchor=south] {$v_2$};
			\draw (3,0) node[anchor=north] {$v_5$};
			\draw (4,0) node[anchor=north] {$v_6$};
			\draw (5,0.5) node[anchor=west] {$v_7$};
			\draw (5,-0.5) node[anchor=west] {$v_n$};
			
			\draw[loosely dotted, line width=1.5pt] (5.00, 0.50) -- (5.00, -0.50);

			
			\draw (C) -- (D);
			\draw (C) -- (A);
			\draw (D) -- (B);
			\draw (A) -- (B);
			\draw (B) -- (G);
			\draw (G) -- (J);
			\draw (J) -- (K);
			\draw (J) -- (L);
			\fill[black] (A) circle (2pt);
			\fill[black] (B) circle (2pt);
			\fill[black] (C) circle (2pt);
			\fill[black] (D) circle (2pt);
			\fill[black] (G) circle (2pt);
			\fill[black] (J) circle (2pt);
			\fill[black] (K) circle (2pt);
			\fill[black] (L) circle (2pt);
		\end{tikzpicture}
		\caption{$G_2(n-6,0,0)$}
	\end{subfigure}
	\caption{Unicyclic graphs $U_3$ and $G_2(n-6,0,0)$.}
	\label{fig_5cycle}
\end{figure}
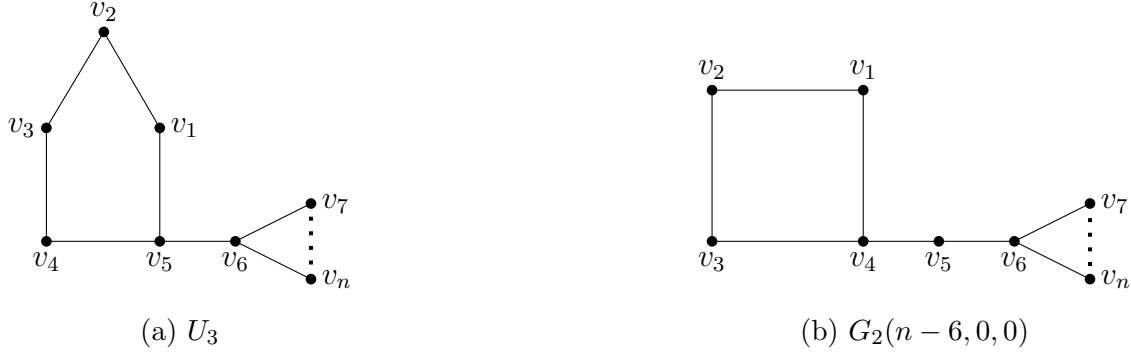
\el
\pf Let $G$ be the graph obtained from $U_3$ by removing the edge $v_1 v_5$ and then adding the edge $v_1v_4$. Then $G\cong G_2(n-6,0,0)$ (see Figure \ref{fig_5cycle}(b)). Let $\mathbf{x}$ be an eigenvector of $A(G)$ corresponding to the eigenvalue $\rho(G)$ and $x_v$ denotes the entry corresponding to the vertex $v$. Note that $\rho(G)>\sqrt{n-5}\geq\sqrt{7}$. Without loss of generality, we assume that $x_{v_1}=1$. By eigen-equations, we get 
\begin{itemize}
	\item[(I1)] $x_{v_2}+x_{v_4}=\rho(G)x_{v_1}$.
	\item[(I2)] $x_{v_1}+x_{v_3}=\rho(G)x_{v_2}$.
	\item[(I3)] $x_{v_4}+x_{v_2}=\rho(G)x_{v_3}$.
	\item[(I4)] $x_{v_1}+x_{v_3}+x_{v_5}=\rho(G)x_{v_4}$.
\end{itemize}
From (I1) and (I3), we have $x_{v_1}=x_{v_3}$. Since $x_{v_1}=x_{v_3}=1$, using (I2), we have $x_{v_2}=\dfrac{2}{\rho(G)}$. Then using the values of $x_{v_2}$ and $x_{v_3}$ in (I3), we have $x_{v_4}=\dfrac{\rho(G)^2-2}{\rho(G)}$. Finally using (I4) and the values of $x_{v_1},x_{v_3}$ and $x_{v_4}$, we have $$x_{v_5}-x_{v_4}=\frac{\rho(G)^3-\rho(G)^2-4\rho(G)+2}{\rho(G)}>0 \text{ (since $\rho(G)>\sqrt{7}$)}.$$

    Since $\mathbf{x}^T(A(U_{4})-A(G))\mathbf{x}=2x_{v_1}(x_{v_5}-x_{v_4})$, we have $\rho(U_{4})>\rho(G)$ by \eqref{eqn:eq_sp_rd}. Thus, by Theorem \ref{thm:min-4-cycle} the result follows.
\qe

We now prove the main result of this subsection.
\bt\label{thm-5-cycle}
 Let $n \ge 12$ be an integer. If $G \in \mathcal{U}_{n,3}$ contains a cycle of length $5$, then $\rho(G) > \rho(U_{n,3}^*)$.
\et

\begin{proof}
    Let us assume that the graph $G \in \mathcal{U}_{n,3}$ and it contains the $5$-cycle $C = [v_1,v_2,v_3,v_4,v_5,v_1]$. Then we have the following observations:
    \begin{itemize}
        \item[(a)] No three consecutive vertices on the cycle can have a tree structure attached to it, otherwise $\gamma > 3$. Let us consider the following example to show it: For simplicity, let us assume single pendant vertices $u,v$ and $w$ are attached to each of the three consecutive vertices of the cycle, say $v_1,v_2$ and $v_3$, respectively. Then we have a matching $\Gamma = \{uv_1, vv_2, wv_3, v_4v_5 \}$ and hence $\gamma \ge 4$. (See (a) of Figure~\ref{fig:5-1-2}).

\begin{figure}[ht]
    \centering
    \begin{subfigure}[b]{0.3\textwidth}
        \centering
        \begin{tikzpicture}[scale=0.8]
  \coordinate (A) at (0.00, 0.00);
  \coordinate (B) at (2.00, 0.00);
  \coordinate (C) at (0.00, 2.00);
  \coordinate (D) at (2.00, 2.00);
  \coordinate (E) at (1.00, 3.50);
  \coordinate (F) at (1.00, 4.50);
  \coordinate (G) at (3.00, 2.00);
  \coordinate (H) at (3.00, 0.00);

  \draw (1.00, 3.50) node[anchor= east] {$v_1$};
  \draw (1.00, 4.50) node[anchor= east] {$u$};
  \draw (2.00, 2.00) node[anchor= south] {$v_2$};
  \draw (3.00, 2.00) node[anchor= south] {$v$};
  \draw (0,2) node[anchor= east] {$v_5$};
  \draw (2.00, 0.00) node[anchor= north] {$v_3$};
  \draw (3.00, 0.00) node[anchor= north] {$w$};
  \draw (0,0) node[anchor= north] {$v_4$};

  \draw (C) -- (E);
  \draw (C) -- (A);
  \draw (A) -- (B);
  \draw (E) -- (D);
  \draw (D) -- (B);
  \draw (E) -- (F);
  \draw (D) -- (G);
  \draw (B) -- (H);
  \fill[black] (A) circle (2pt);
  \fill[black] (B) circle (2pt);
  \fill[black] (C) circle (2pt);
  \fill[black] (D) circle (2pt);
  \fill[black] (E) circle (2pt);
  \fill[black] (F) circle (2pt);
  \fill[black] (G) circle (2pt);
  \fill[black] (H) circle (2pt);
\end{tikzpicture}
        \caption{$U_4$}
    \end{subfigure}
    \hfill
    \begin{subfigure}[b]{0.3\textwidth}
        \centering
        \begin{tikzpicture}[scale=0.8]
  \coordinate (A) at (0.00, 0.00);
  \coordinate (B) at (2.00, 0.00);
  \coordinate (C) at (0.00, 2.00);
  \coordinate (D) at (2.00, 2.00);
  \coordinate (E) at (1.00, 3.50);
  \coordinate (G) at (3.00, 2.00);
  \coordinate (H) at (3.00, 0.00);
  \coordinate (I) at (4.00, 0);

  \draw (1.00, 3.50) node[anchor= east] {$v_1$};
  \draw (2.00, 2.00) node[anchor= south] {$v_2$};
  \draw (3.00, 2.00) node[anchor= south] {$w$};
  \draw (0,2) node[anchor= east] {$v_5$};
  \draw (2.00, 0.00) node[anchor= north] {$v_3$};
  \draw (3.00, 0.00) node[anchor= north] {$v$};
  \draw (4.00, 0.00) node[anchor= north] {$u$};
  \draw (0,0) node[anchor= north] {$v_4$};

  \draw (C) -- (E);
  \draw (C) -- (A);
  \draw (A) -- (B);
  \draw (E) -- (D);
  \draw (D) -- (B);
  \draw (D) -- (G);
  \draw (B) -- (H);
  \draw (H) -- (I);

  \fill[black] (A) circle (2pt);
  \fill[black] (B) circle (2pt);
  \fill[black] (C) circle (2pt);
  \fill[black] (D) circle (2pt);
  \fill[black] (E) circle (2pt);
  \fill[black] (G) circle (2pt);
  \fill[black] (H) circle (2pt);
  \fill[black] (I) circle (2pt);
\end{tikzpicture}
        \caption{$U_5$}
    \end{subfigure}
    \hfill
    \begin{subfigure}[b]{0.3\textwidth}
        \centering
        \begin{tikzpicture}[scale=0.8]
  \coordinate (A) at (0.00, 0.00);
  \coordinate (B) at (2.00, 0.00);
  \coordinate (C) at (0.00, 2.00);
  \coordinate (D) at (2.00, 2.00);
  \coordinate (E) at (1.00, 3.50);
  \coordinate (H) at (3.00, 0.00);
  \coordinate (F) at (4.00, 0.00);
  \coordinate (G) at (5.00, 0.0);

  \draw (1.00, 3.50) node[anchor= east] {$v_1$};
  \draw (2.00, 2.00) node[anchor= west] {$v_2$};
  \draw (0,2) node[anchor= east] {$v_5$};
  \draw (2.00, 0.00) node[anchor= north] {$v_3$};
  \draw (3.00, 0.00) node[anchor= north] {$w$};
  \draw (4.00, 0.00) node[anchor= north] {$v$};
  \draw (5.00, 0.00) node[anchor= north] {$u$};
  \draw (0,0) node[anchor= north] {$v_4$};

  \draw (C) -- (E);
  \draw (C) -- (A);
  \draw (A) -- (B);
  \draw (E) -- (D);
  \draw (D) -- (B);
  \draw (B) -- (H);
  \draw (H) -- (F);
  \draw (F) -- (G);
  
  \fill[black] (A) circle (2pt);
  \fill[black] (B) circle (2pt);
  \fill[black] (C) circle (2pt);
  \fill[black] (D) circle (2pt);
  \fill[black] (E) circle (2pt);
  \fill[black] (H) circle (2pt);
  \fill[black] (F) circle (2pt);
  \fill[black] (G) circle (2pt);
\end{tikzpicture}
        \caption{$U_6$}
    \end{subfigure}
    \caption{Unicyclic graphs with a $5$-cycle and $\gamma >3$.}
    \label{fig:5-1-2}
\end{figure}
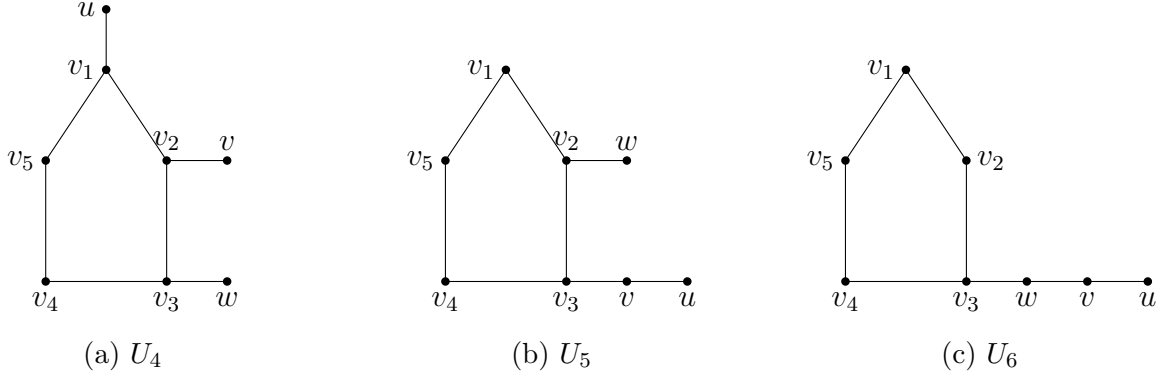
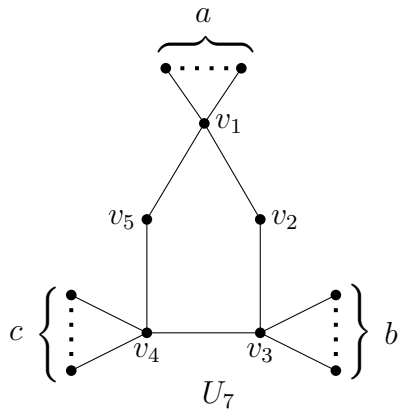
        \item[(b)] There can be at most one quasi-pendant vertex attached to a vertex of the $5$-cycle and if there is a quasi-pendant vertex attached to a vertex of the $5$-cycle, then there cannot be any tree structure attached to other vertices of the cycle, otherwise $\gamma > 3$.  Let us consider the following example to show it: For simplicity, let us assume a single quasi-pendant vertex $v$ is attached to $v_3$ and a pendant vertex $u$ is attached to $v$ and another pendant vertex $w$ is attached to $v_2$. Then we have a matching $\Gamma = \{uv, wv_2, v_3v_4, v_1v_5 \}$ and hence $\gamma \ge 4$. A similar argument follows, if we have a tree structure on any other vertex of the cycle, or we have more than one quasi-pendant vertex attached to the same vertex of the cycle (See (b) of Figure~\ref{fig:5-1-2}).
        \item[(c)] If $T$ is a tree branch in $G$ attached to the vertex $v_i$ ($i\in \{1,2,3,4,5\}$) of $C$, then $d(v_i,v)\leq 2$ for every $v\in V(T)$. Otherwise $\gamma > 3$. To illustrate this, suppose that the tree branch attached to $v_3$ is the path $[u,v,w,v_3]$, where $u$ is a pendant vertex. Then we have a matching $\Gamma = \{uv, wv_3, v_1v_2, v_4v_5\}$ and hence $\gamma \ge 4$. (See (c) of Figure~\ref{fig:5-1-2}).
    \end{itemize}
    Combining the above observations, we can conclude that there are only two possible structures $U_3$ (refer to Figure \ref{fig_5cycle}(a)) and $U_7$ (refer to Figure~\ref{fig:5-1-1}). Note that $U_7$ is obtained from the cycle $C=[v_1,v_2,v_3,v_4,v_5,v_1]$ by attaching $a,b,c$ pendant vertices at $v_1,v_3,v_4$, respectively. Since at least one among $a,b,c$ is positive, we assume that $a$ is positive.
    \begin{figure}[ht]
    \centering
    \begin{tikzpicture}[scale=1]
  \coordinate (A) at (0.00, 0.00);
  \coordinate (B) at (1.50, 0.00);
  \coordinate (C) at (0.00, 1.50);
  \coordinate (D) at (1.50, 1.50);
  \coordinate (E) at (0.76, 2.77);
  \coordinate (F) at (2.50, 0.50);
  \coordinate (G) at (2.50, -0.50);
  \coordinate (H) at (-1.00, 0.50);
  \coordinate (I) at (-1.00, -0.50);
  \coordinate (J) at (0.25, 3.50);
  \coordinate (K) at (1.25, 3.50);

  \draw (0.76, 2.77) node[anchor= west] {$v_1$};
  \draw (1.5,1.5) node[anchor= west] {$v_2$};
  \draw (0,1.5) node[anchor= east] {$v_5$};
  \draw (1.5, 0.00) node[anchor= north] {$v_3$};
  \draw (0,0) node[anchor= north] {$v_4$};
  \draw (0.95,-0.5) node[anchor= north] {$U_7$};

  \draw (C) -- (E);
  \draw (C) -- (A);
  \draw (A) -- (B);
  \draw (E) -- (D);
  \draw (D) -- (B);
  \draw (J) -- (E);
  \draw (E) -- (K);
  \draw (H) -- (A);
  \draw (I) -- (A);
  \draw (B) -- (F);
  \draw (B) -- (G);
  \fill[black] (A) circle (2pt);
  \fill[black] (B) circle (2pt);
  \fill[black] (C) circle (2pt);
  \fill[black] (D) circle (2pt);
  \fill[black] (E) circle (2pt);
  \fill[black] (F) circle (2pt);
  \fill[black] (G) circle (2pt);
  \fill[black] (H) circle (2pt);
  \fill[black] (I) circle (2pt);
  \fill[black] (J) circle (2pt);
  \fill[black] (K) circle (2pt);

  \draw[loosely dotted, line width=1.5pt] (2.50, -0.50) -- (2.50, 0.50);
  \draw[loosely dotted, line width=1.5pt] (-1, -0.50) -- (-1, 0.50);
  \draw[loosely dotted, line width=1.5pt] (0.25, 3.50) -- (1.25, 3.50);

  \draw (2.5,0.8) node[anchor=north west, rotate=0] {$\Biggl\}$};
  \draw (3,0) node[anchor=west] {$b$};

  \draw (-1,-0.8) node[anchor=north west, rotate=180] {$\Biggl\}$};
  \draw (-1.5,0) node[anchor=east] {$c$};

  \draw (0,3.5) node[anchor=north west, rotate=90] {$\Biggl\}$};
  \draw (1,4.2) node[anchor=east] {$a$};
\end{tikzpicture}
\caption{Unicyclic graphs with a $5$-cycle and $\gamma = 3$}
    \label{fig:5-1-1}
\end{figure}

If $G\cong U_3$, then Lemma \ref{lem:5-4-comp} implies the required result. Thus, assume that $G\cong U_7$. Let $G^*$ be the graph obtained from $G$ by subdividing the edge $v_3 v_4$, then using Lemma~\ref{lem-sub}, we have that $\rho(G^*) < \rho(G)$. Next, $G^{**}$ be the graph obtained from $G^*$ by deleting one pendant vertex, the using Lemma~\ref{lem:subgraph} we have that $\rho(G^{**}) < \rho(G^*) < \rho(G)$. Observe that $G^{**}\cong G_3(a-1,b,c)$ (see Figure~\ref{fig:6-2}). Hence using Lemmas~\ref{lem-6-cycle} and \ref{lem:6-4-comp}, we have $\rho(U_{n,3}^*) < \rho(G^{**}) < \rho(U_7)$.
\end{proof}


\subsection{Unicyclic graphs with girth $4$ and $\gamma=3$}
In this subsection, we compare the spectral radius of $U_{n,3}^*$ with the spectral radii of graphs in $\mathcal U_{n,3}$ whose unique cycle is a $4$-cycle. Before that, let us define the following class of graphs:
\bd\label{def:cycle-4-2} Let $G_4$ be a graph obtained from a path $[v_1,v_2,v_3,v_4,v_5]$ on $5$ vertices by adding an edge between the vertices $v_2$ and $v_5$. Let $a,b,c$ be nonnegative integers such that $b>0$ and $a+b+c+6=n\geq 12$. Define $G_4(a,b,c)$ be the graph obtained from $H$ by attaching $a,b,c$ pendant vertices to $v_1,v_3,v_5$, respectively {\rm(}see Figure \ref{fig:comp4_4}{\rm)}.
\begin{figure}[ht]
		\centering
		\begin{tikzpicture}[scale=1]
			\coordinate (A) at (0.00, 0.00);
			\coordinate (B) at (2.00, 0.00);
			\coordinate (C) at (0.00, 2.00);
			\coordinate (D) at (2.00, 2.00);
			\coordinate (E) at (-0.97, 2.27);
			\coordinate (F) at (-0.26, 2.97);
			\coordinate (H) at (1.50, -1.00);
			\coordinate (I) at (2.50, -1.00);
			\coordinate (G) at (3.00, 2.00);
			\coordinate (J) at (4.00, 2.50);
			\coordinate (K) at (4.00, 1.50);
			\coordinate (L) at (2.90, -0.80);

			\draw (0,0) node[anchor=north] {$v_4$};
			\draw (2,0.4) node[anchor=north west] {$v_3$};
			\draw (2,2) node[anchor= south] {$v_2$};
			\draw (0,2) node[anchor=south west] {$v_5$};
			\draw (3,2) node[anchor= south] {$v_1$};
			
			\draw[loosely dotted, line width=1.5pt] (1.50, -1.00) -- (2.50, -1.00);
			\draw[loosely dotted, line width=1.5pt] (4.00, 1.50) -- (4.00, 2.50);
			\draw[loosely dotted, line width=1.5pt] (-0.97, 2.27) -- (-0.26, 2.97);
			
			\draw (C) -- (D);
			\draw (C) -- (A);
			\draw (D) -- (B);
			\draw (A) -- (B);
			\draw (E) -- (C);
			\draw (F) -- (C);
			\draw (B) -- (H);
			\draw (B) -- (I);
			\draw (D) -- (G);
			\draw (G) -- (J);
			\draw (G) -- (K);
			\fill[black] (A) circle (2pt);
			\fill[black] (B) circle (2pt);
			\fill[black] (C) circle (2pt);
			\fill[black] (D) circle (2pt);
			\fill[black] (E) circle (2pt);
			\fill[black] (F) circle (2pt);
			\fill[black] (H) circle (2pt);
			\fill[black] (I) circle (2pt);
			\fill[black] (G) circle (2pt);
			\fill[black] (J) circle (2pt);
			\fill[black] (K) circle (2pt);

        \draw (4,2.8) node[anchor=north west, rotate=0] {$\Biggl\}$};
        \draw (4.5,2) node[anchor=west] {$b$};

        \draw (2.75,-0.9) node[anchor=north west, rotate=-90] {$\Biggl\}$};
        \draw (2.2,-1.7) node[anchor=east] {$a$};

    \draw (-1.1,2) node[anchor=north west, rotate=135] {$\Biggl\}$};
    \draw (-0.8,2.9) node[anchor=east] {$c$};
		\end{tikzpicture}
		\caption{The unicyclic graph $G_4(a,b,c)$.}
	\label{fig:comp4_4}
\end{figure}
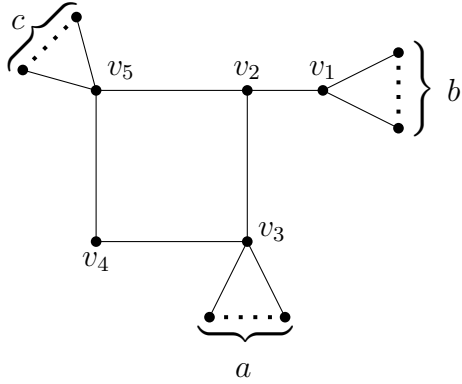
\ed
The following result is needed to prove the main result of this subsection.
\begin{lem}\label{lem-4-cycle-comp}
	Let $a$, $b$, and $c$ be nonnegative integers such that $n=a+b+c+6\geq 12$. Suppose that $G_2(a,b,c)$ and $G_4(a,b+1,c)$ are graphs on $n$ vertices as defined in Definitions \ref{def:cycle-4-1} and \ref{def:cycle-4-2}, respectively. Then $\rho(G_2(a,b,c))<\rho(G_4(a,b+1,c))$.
\end{lem}

\begin{proof}
	One can calculate that $\rho(G_4(a,b+1,c))^2$ is the largest root of the monic cubic polynomial $$g(\lbd)=\lbd^3-(a+b+c+6)\lbd^2+(ab+bc+ac+5a+3b+4c+5)\lbd-(abc+2ab+3ac+bc+2a+b+2c+1).$$
	Thus, $\psi_\lbd(a,b,c)-g(\lbd)=x-(b-1)$. Since $\rho(G_4(a,b+1,c))^2>b+3$, Lemma \ref{poly}(i) implies that $\rho(G_2(a,b,c))^2<\rho(G_4(a,b+1,c))^2$. This completes the proof.
\end{proof}

We now prove the main result of this subsection that shows that among all graphs in $\mathcal{U}_{n,3}$ with cycle length $4$, the spectral radius of the graph $U_{n,3}^*$ is the minimum.
\bt\label{main_1}
	Let $n \ge 12$ be an integer. If $G \in \mathcal{U}_{n,3}$ has a cycle of length $4$, then $\rho(G) \ge \rho(U_{n,3}^*)$ with equality if and only if $G\cong U_{n,3}^*$.
\et

\begin{proof}
	Let $G \in \mathcal{U}_{n,3}$ be a graph which contains a $4$-cycle $C=[v_1,v_2,v_3,v_4,v_1]$ (say). Let $\mathcal{M}$ be the collection of all maximal matchings of $G$. Next, based on the maximal matching of $G$, we divide into the following cases:
	\begin{itemize}
		\item[Case 1.] For all $\Gamma \in \mathcal{M}$, every edge in $\Gamma$ has at least one endpoint on the $4$-cycle. In this case, we have only pendant vertices attached to the vertices of the cycle. Since $\gamma = 3$, we must have pendant vertices attached to at least $2$ vertices of the cycle, otherwise $\gamma = 2$. Also, observe that if pendant vertices are attached to exactly two vertices of the cycle, then the vertices on the cycle must be adjacent. Finally, there can be pendant vertices in at most $3$ vertices of the $4$-cycle, otherwise $\gamma > 3$. Hence $G\cong U_8$ (see Figure~\ref{fig:4-1}). Note that $U_8$ is obtained from the cycle $C=[v_1,v_2,v_3,v_4,v_1]$ by attaching $a$, $b$, and $c$ pendant vertices to $v_1$, $v_2$, and $v_3$, respectively, where $b>0$ and at least one among $a$ and $c$ is nonzero. Without loss of generality assume that $c>0$.
		\begin{figure}[ht]
			\centering
				\begin{tikzpicture}[scale=1]
					\coordinate (A) at (0.00, 0.00);
					\coordinate (B) at (2.00, 0.00);
					\coordinate (C) at (0.00, 2.00);
					\coordinate (D) at (2.00, 2.00);
					\coordinate (E) at (2.00, 3.00);
					\coordinate (F) at (3.00, 2.00);
					\coordinate (G) at (3.00, 0.00);
					\coordinate (H) at (2.00, -1.00);
					\coordinate (I) at (-1.00, 0.00);
					\coordinate (J) at (0.00, -1.00);
					
					\draw (0,0) node[anchor=north west] {$v_1$};
					\draw (2,0) node[anchor=north east] {$v_2$};
					\draw (2,2) node[anchor= south east] {$v_3$};
					\draw (0,2) node[anchor=south west] {$v_4$};
					
					\draw[loosely dotted, line width=1.5pt] (3.00, 0.00) -- (2.00, -1.00);
					\draw[loosely dotted, line width=1.5pt] (-1.00, 0.00) -- (0.00, -1.00);
					\draw[loosely dotted, line width=1.5pt] (2.00, 3.00) -- (3.00, 2.00);\draw (0.5,-1) node[anchor=north west] {$U_8$};
					
					\draw (J) -- (A);
					\draw (C) -- (D);
					\draw (C) -- (A);
					\draw (D) -- (B);
					\draw (A) -- (B);
					\draw (I) -- (A);
					\draw (D) -- (E);
					\draw (D) -- (F);
					\draw (B) -- (G);
					\draw (B) -- (H);
					\fill[black] (A) circle (2pt);
					\fill[black] (B) circle (2pt);
					\fill[black] (C) circle (2pt);
					\fill[black] (D) circle (2pt);
					\fill[black] (E) circle (2pt);
					\fill[black] (F) circle (2pt);
					\fill[black] (G) circle (2pt);
					\fill[black] (H) circle (2pt);
					\fill[black] (I) circle (2pt);
					\fill[black] (J) circle (2pt);
				\end{tikzpicture}
			\caption{A unicyclic graph with $4$-cycle and $\gamma = 3$.}
			\label{fig:4-1}
		\end{figure}
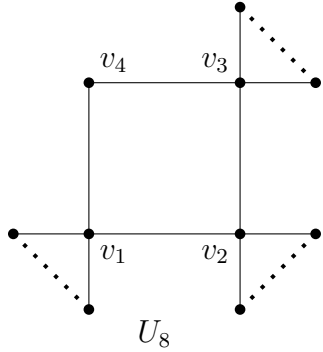
		Let $G^*$ be the graph obtained from $G$ by subdividing the edge $v_2v_3$. Then by Lemma~\ref{lem-sub}, we have $\rho(G^*) < \rho (G)$. If $G^{**}$ is a subgraph of $G^*$ obtained by deleting a pendent vertex, then  Lemma~\ref{lem:subgraph} implies $\rho(G^{**}) < \rho(G^*) < \rho (G)$. Note that the graph $G^{**}\in \mathcal{U}_{n,3}$ and it contains a $5$ cycle. Thus, by Theorem~\ref{thm-5-cycle}, we have that $ \rho(U_{n,3}^*)<\rho(G^{**})$. Thus, $\rho (U_8) > \rho(U_{n,3}^*)$.
		\item[Case 2.] For all $\Gamma \in \mathcal{M}$, there exists at least one edge in $\Gamma$ such that both of its endpoints are not part of the $4$-cycle. Let the edge be $uv$. Then $G-u-v$ consists of a graph containing a $4$-cycle, possibly together with some isolated vertices ($K_1$'s), and has $\gamma=2$. Thus, it follows from Theorem~\ref{thm-gamma-2} that the nontrivial component of $G - u - v$ is isomorphic to $G_1(b,c)$ (see Figure~\ref{fig:gamma-2-4}) for some nonnegative integers $b$ and $c$. Without loss of generality, let us assume that the vertex $u$ is closer to the vertices of the cycle than $v$. Then $\min\{ d(u,v_1),d(u,v_2),d(u,v_3),d(u,v_4)\} \le 2$, as otherwise $\gamma \ge 3$. Hence, the only possible structural configurations of $G$ are $U_9$, $U_{10}$, and $U_{11}$ (see Figure~\ref{fig:4-2}).
		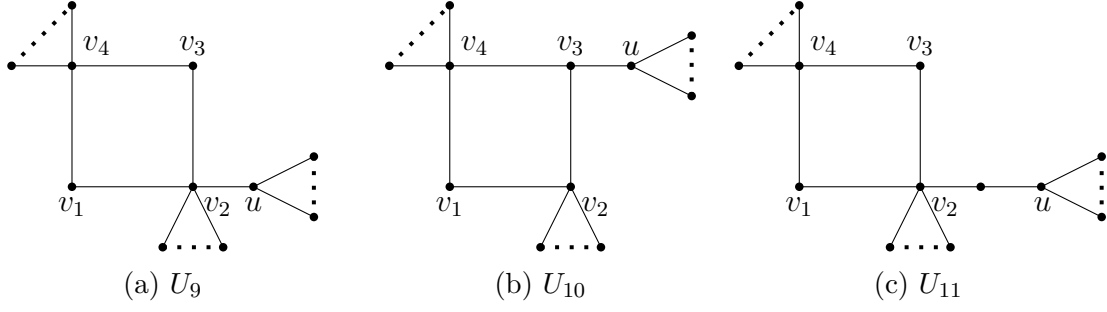
\begin{figure}[ht]
			\centering
			\begin{subfigure}[b]{0.3\textwidth}
				\centering
				\begin{tikzpicture}[scale=0.8]
					\coordinate (A) at (0.00, 0.00);
					\coordinate (B) at (2.00, 0.00);
					\coordinate (C) at (0.00, 2.00);
					\coordinate (D) at (2.00, 2.00);
					\coordinate (G) at (3.00, 0.00);
					\coordinate (E) at (-1.00, 2.00);
					\coordinate (F) at (0.00, 3.00);
					\coordinate (H) at (1.50, -1.00);
					\coordinate (I) at (2.50, -1.00);
					\coordinate (J) at (4.00, -0.50);
					\coordinate (K) at (4.00, 0.50);
					
					\draw (0,0) node[anchor=north] {$v_1$};
					\draw (2,0) node[anchor=north west] {$v_2$};
					\draw (2,2) node[anchor= south] {$v_3$};
					\draw (0,2) node[anchor=south west] {$v_4$};
					\draw (3,0) node[anchor=north] {$u$};
					
					\draw[loosely dotted, line width=1.5pt] (1.50, -1.00) -- (2.50, -1.00);
					\draw[loosely dotted, line width=1.5pt] (4.00, 0.50) -- (4.00, -0.50);
					\draw[loosely dotted, line width=1.5pt] (-1.00, 2.00) -- (0.00, 3.00);
					
					\draw (C) -- (D);
					\draw (C) -- (A);
					\draw (D) -- (B);
					\draw (A) -- (B);
					\draw (B) -- (G);
					\draw (E) -- (C);
					\draw (F) -- (C);
					\draw (B) -- (H);
					\draw (B) -- (I);
					\draw (G) -- (K);
					\draw (G) -- (J);
					\fill[black] (A) circle (2pt);
					\fill[black] (B) circle (2pt);
					\fill[black] (C) circle (2pt);
					\fill[black] (D) circle (2pt);
					\fill[black] (G) circle (2pt);
					\fill[black] (E) circle (2pt);
					\fill[black] (F) circle (2pt);
					\fill[black] (H) circle (2pt);
					\fill[black] (I) circle (2pt);
					\fill[black] (J) circle (2pt);
					\fill[black] (K) circle (2pt);
				\end{tikzpicture}
				\caption{$U_9$}
			\end{subfigure}
			\begin{subfigure}[b]{0.3\textwidth}
				\centering
				\begin{tikzpicture}[scale=0.8]
					\coordinate (A) at (0.00, 0.00);
					\coordinate (B) at (2.00, 0.00);
					\coordinate (C) at (0.00, 2.00);
					\coordinate (D) at (2.00, 2.00);
					\coordinate (E) at (-1.00, 2.00);
					\coordinate (F) at (0.00, 3.00);
					\coordinate (H) at (1.50, -1.00);
					\coordinate (I) at (2.50, -1.00);
					\coordinate (G) at (3.00, 2.00);
					\coordinate (J) at (4.00, 2.50);
					\coordinate (K) at (4.00, 1.50);
					
					\draw (0,0) node[anchor=north] {$v_1$};
					\draw (2,0) node[anchor=north west] {$v_2$};
					\draw (2,2) node[anchor= south] {$v_3$};
					\draw (0,2) node[anchor=south west] {$v_4$};
					\draw (3,2) node[anchor= south] {$u$};
					
					\draw[loosely dotted, line width=1.5pt] (1.50, -1.00) -- (2.50, -1.00);
					\draw[loosely dotted, line width=1.5pt] (4.00, 1.50) -- (4.00, 2.50);
					\draw[loosely dotted, line width=1.5pt] (-1.00, 2.00) -- (0.00, 3.00);
					
					\draw (C) -- (D);
					\draw (C) -- (A);
					\draw (D) -- (B);
					\draw (A) -- (B);
					\draw (E) -- (C);
					\draw (F) -- (C);
					\draw (B) -- (H);
					\draw (B) -- (I);
					\draw (D) -- (G);
					\draw (G) -- (J);
					\draw (G) -- (K);
					\fill[black] (A) circle (2pt);
					\fill[black] (B) circle (2pt);
					\fill[black] (C) circle (2pt);
					\fill[black] (D) circle (2pt);
					\fill[black] (E) circle (2pt);
					\fill[black] (F) circle (2pt);
					\fill[black] (H) circle (2pt);
					\fill[black] (I) circle (2pt);
					\fill[black] (G) circle (2pt);
					\fill[black] (J) circle (2pt);
					\fill[black] (K) circle (2pt);
				\end{tikzpicture}
				\caption{$U_{10}$}
			\end{subfigure}
			\begin{subfigure}[b]{0.3\textwidth}
				\centering
				\begin{tikzpicture}[scale=0.8]
					\coordinate (A) at (0.00, 0.00);
					\coordinate (B) at (2.00, 0.00);
					\coordinate (C) at (0.00, 2.00);
					\coordinate (D) at (2.00, 2.00);
					\coordinate (G) at (3.00, 0.00);
					\coordinate (E) at (-1.00, 2.00);
					\coordinate (F) at (0.00, 3.00);
					\coordinate (H) at (1.50, -1.00);
					\coordinate (I) at (2.50, -1.00);
					\coordinate (J) at (4.00, 0.00);
					\coordinate (K) at (5.00, 0.50);
					\coordinate (L) at (5.00, -0.50);
					
					\draw (0,0) node[anchor=north] {$v_1$};
					\draw (2,0) node[anchor=north west] {$v_2$};
					\draw (2,2) node[anchor= south] {$v_3$};
					\draw (0,2) node[anchor=south west] {$v_4$};\draw (3.7,0) node[anchor=north west] {$u$};
					
					\draw[loosely dotted, line width=1.5pt] (1.50, -1.00) -- (2.50, -1.00);
					\draw[loosely dotted, line width=1.5pt] (5.00, 0.50) -- (5.00, -0.50);
					\draw[loosely dotted, line width=1.5pt] (-1.00, 2.00) -- (0.00, 3.00);
					
					\draw (C) -- (D);
					\draw (C) -- (A);
					\draw (D) -- (B);
					\draw (A) -- (B);
					\draw (B) -- (G);
					\draw (E) -- (C);
					\draw (F) -- (C);
					\draw (B) -- (H);
					\draw (B) -- (I);
					\draw (G) -- (J);
					\draw (J) -- (K);
					\draw (J) -- (L);
					\fill[black] (A) circle (2pt);
					\fill[black] (B) circle (2pt);
					\fill[black] (C) circle (2pt);
					\fill[black] (D) circle (2pt);
					\fill[black] (G) circle (2pt);
					\fill[black] (E) circle (2pt);
					\fill[black] (F) circle (2pt);
					\fill[black] (H) circle (2pt);
					\fill[black] (I) circle (2pt);
					\fill[black] (J) circle (2pt);
					\fill[black] (K) circle (2pt);
					\fill[black] (L) circle (2pt);
				\end{tikzpicture}
				\caption{$U_{11}$}
			\end{subfigure}
			\caption{Unicyclic graphs with $4$-cycle and $\gamma = 3$.}
			\label{fig:4-2}
		\end{figure}

Suppose $G \cong U_9$ (see (a) of Figure~\ref{fig:4-2}). Let $G^*$ be the graph obtained from $G$ by subdividing the edge $uv_2$. Then by Lemma~\ref{lem-sub}, we have $\rho(G^*) < \rho (G)$. Further, if $G^{**}$ is a subgraph of $G^*$ obtained by deleting a pendent edge, then using Lemma~\ref{lem:subgraph}, we have $\rho(G^{**}) < \rho(G^*) < \rho (G)$. Note that the graph $G^{**}$ is isomorphic to the graph $G_2(a,b,c)$ for some nonnegative integers $a,b,c$ (see Definition \ref{def:cycle-4-1}). Thus, by Theorem \ref{thm:min-4-cycle}, we have $\rho(U_9) > \rho(U_{n,3}^*)$.

Suppose $G \cong U_{10}$ (see (b) of Figure~\ref{fig:4-2}). Since $U_{10}\cong G_4(a,b,c)$ for some nonnegative integers $a,b,c$ (see Definition \ref{def:cycle-4-2}), we have $\rho(U_{10}) > \rho(U_{n,3}^*)$ (by Lemma \ref{lem-4-cycle-comp}).

Suppose $G \cong U_{11}$ (see (c) of Figure~\ref{fig:4-2}). Then $G\cong G_2(a,b,c)$ for some nonnegative integers $a,b,c$ (see Definition \ref{def:cycle-4-1}). By Theorem \ref{thm:min-4-cycle}, we have $\rho(U_{11}) \geq \rho(U_{n,3}^*)$ with equality if and only if $G\cong U_{n,3}^*$

\end{itemize}
	
This completes the proof.
\end{proof}


\subsection{Unicyclic graphs with girth $3$ and $\gamma=3$}
In this subsection, we compare the spectral radius of $U_{n,3}^*$ with the spectral radii of graphs in $\mathcal U_{n,3}$ whose unique cycle is a $4$-cycle. Before that, let us define the following two classes of graphs:
\bd\label{def:cycle-4-2-1} Let $G_5$ be a graph obtained from a path $[v_1,v_2,v_3,v_4,v_5,v_6]$ on $6$ vertices by adding an edge between the vertices $v_1$ and $v_3$. Let $n$ and $r$ be two integers such that $n\geq 12$ and $1\leq r\leq n-7$. Define $G_5(r,n-6-r)$ to be the graph on $n$ vertices obtained from $G_5$ by attaching $r$ pendant vertices to $v_4$ and $n-r-6$ pendant vertices to $v_6$ {\rm(}see Figure \ref{fig:3-2-3-1}{\rm)}.
\begin{figure}[ht]
    \centering
        \begin{tikzpicture}[scale=1]
  \coordinate (A) at (0.00, 0.00);
  \coordinate (B) at (2.00, 0.00);
  \coordinate (C) at (1.00, 2.00);
  \coordinate (F) at (3.00, 0.00);
  \coordinate (D) at (4.00, 0.00);
  \coordinate (G) at (2.50, -1.00);
  \coordinate (H) at (3.50, -1.00);
  \coordinate (E) at (5.00, 0.00);
  \coordinate (I) at (4.50, -1.00);
  \coordinate (J) at (5.50, -1.00);

    \draw (0,0) node[anchor=north] {$v_2$};
    \draw (2,0) node[anchor=north] {$v_3$};
    \draw (1,2) node[anchor=west] {$v_1$};
    \draw (3,0) node[anchor=south] {$v_4$};
    \draw (4,0) node[anchor=south] {$v_5$};
    \draw (5,0) node[anchor=south] {$v_6$};

  \draw (A) -- (B);
  \draw (A) -- (C);
  \draw (C) -- (B);
  \draw (B) -- (F);
  \draw (F) -- (D);
  \draw (F) -- (G);
  \draw (F) -- (H);
  \draw (D) -- (E);
  \draw (E) -- (I);
  \draw (E) -- (J);
  \fill[black] (A) circle (2pt);
  \fill[black] (B) circle (2pt);
  \fill[black] (C) circle (2pt);
  \fill[black] (F) circle (2pt);
  \fill[black] (D) circle (2pt);
  \fill[black] (G) circle (2pt);
  \fill[black] (H) circle (2pt);
  \fill[black] (E) circle (2pt);
  \fill[black] (I) circle (2pt);
  \fill[black] (J) circle (2pt);

  \draw[loosely dotted, line width=1.5pt] (3.50, -1.00) -- (2.50, -1.00);
  \draw[loosely dotted, line width=1.5pt] (4.50, -1.00) -- (5.50, -1.00);

  \draw (3.8,-1) node[anchor=north west, rotate=-90] {$\Biggl\}$};
        \draw (3.3,-1.7) node[anchor=east] {$r$};

    \draw (5.8,-1) node[anchor=north west, rotate=-90] {$\Biggl\}$};
        \draw (6,-1.7) node[anchor=east] {$n-6-r$};
\end{tikzpicture}
    \caption{The unicyclic graph $G_5(r,n-6-r)$.}
    \label{fig:3-2-3-1}
\end{figure}
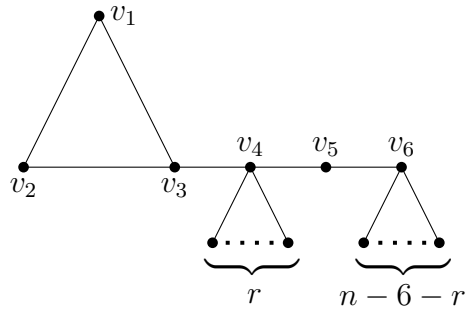
\ed

\bd\label{def:cycle-4-2-2} Let $G_6$ be a graph obtained from a path $[v_1,v_2,v_3,v_4,v_5]$ on $5$ vertices by adding an edge between the vertices $v_2$ and $v_4$. Let $n$ and $r$ be two integers such that $n\geq 12$ and $r\leq n-7$. Define $G_6(r,n-5-r)$ to be the graph on $n$ vertices obtained from $G_6$ by attaching $r$ pendant vertices to $v_1$ and $n-r-5$ pendant vertices to $v_5$ {\rm(}see Figure \ref{fig:3-2-3-2}{\rm)}.
\begin{figure}[ht]
    \centering
        \begin{tikzpicture}[scale=1]
  \coordinate (A) at (0.00, 0.00);
  \coordinate (B) at (2.00, 0.00);
  \coordinate (C) at (1.00, 2.00);
  \coordinate (F) at (3.00, 0.00);
  \coordinate (G) at (4.00, 0.50);
  \coordinate (H) at (4.00, -0.50);
  \coordinate (D) at (-1.00, 0.00);
  \coordinate (I) at (-2.00, 0.50);
  \coordinate (J) at (-2.00, -0.50);

    \draw (0,0) node[anchor=north] {$v_2$};
    \draw (2,0) node[anchor=north] {$v_4$};
    \draw (1,2) node[anchor=west] {$v_3$};
    \draw (3,0) node[anchor=north] {$v_5$};
    \draw (-1,0) node[anchor=north] {$v_1$};

  \draw (A) -- (B);
  \draw (A) -- (C);
  \draw (C) -- (B);
  \draw (B) -- (F);
  \draw (B) -- (D);
  \draw (F) -- (G);
  \draw (F) -- (H);
  \draw (I) -- (D);
  \draw (J) -- (D);
  \fill[black] (A) circle (2pt);
  \fill[black] (B) circle (2pt);
  \fill[black] (C) circle (2pt);
  \fill[black] (F) circle (2pt);
  \fill[black] (G) circle (2pt);
  \fill[black] (H) circle (2pt);
  \fill[black] (D) circle (2pt);
  \fill[black] (I) circle (2pt);
  \fill[black] (J) circle (2pt);

  \draw[loosely dotted, line width=1.5pt] (-2.00, 0.50) -- (-2.00, -0.50);
  \draw[loosely dotted, line width=1.5pt] (4.00, 0.50) -- (4.00, -0.50);

  \draw (4,0.8) node[anchor=north west, rotate=0] {$\Biggl\}$};
  \draw (4.5,0) node[anchor=west] {$n-5-r$};

  \draw (-2.0,-0.8) node[anchor=north west, rotate=180] {$\Biggl\}$};
  \draw (-2.5,0) node[anchor=east] {$r$};
\end{tikzpicture}
    \caption{The unicyclic graph $G_6(r,n-5-r)$.}
    \label{fig:3-2-3-2}
\end{figure}
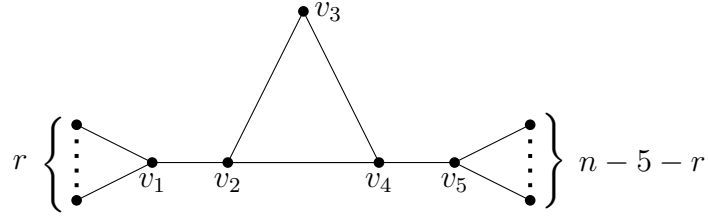
\ed

We now proceed to find the value of $r$ such that $\rho(G_5(r,n-r-6))$ is minimum.

\begin{lem}\label{lem:U-19-min}
	Let $r$ and $n$ be a positive integer such that $n\geq 12$ and $r\leq n-7$. Suppose that $G_5(r,n-r-6)$ is the graph as defined in Definition \ref{def:cycle-4-2-1}. Then $\rho(G_5(r,n-r-6)\geq\rho(G_5(\lfloor\frac{n-7}{2}\rfloor,\lceil\frac{n-5}{2}\rceil))$ with equality if and only if $r=\lfloor\frac{n-7}{2}\rfloor$.
\end{lem}
\pf One can verify that the characteristic polynomial of $G_5(r,n-r-6)$ is 
\begin{align*}
	\phi_\lbd(G_5(r,n-r-6))=&\lbd^{n-7}(\lbd^7-n\lbd^6-2\lbd^4+(-r^2+(n-7)r+5n-22)\lbd^3+2(n-4)\lbd^2\\
	&-(-3r^2+(3n-19)r+4n-23)\lbd-2(-r^2+(n-6)r+n-6))\\
	=&\lbd^{n-7}f(\lbd) \text{ (say)}.
\end{align*}
Similarly, one can calculate that $\phi_\lbd\left(G_5\left(\lfloor\frac{n-7}{2}\rfloor,\lceil\frac{n-5}{2}\rceil\right)\right)=\lbd^{n-7}g(\lbd)$, where 
$$g(\lbd)=
\begin{cases}
	\begin{aligned}
		&\lbd^7-(2k-1)\lbd^5-2\lbd^4+(k^2+2k-11)\lbd^3+2(2k-5)\lbd^2\\
		&-(3k^2-14k+13)\lbd-2(k^2-5k+5),
	\end{aligned}
	& \text{if } n=2k-1,\\[10pt]
	
	\begin{aligned}
		&\lbd^7-2k\lbd^5-2\lbd^4+(k^2+3k-10)\lbd^3+2(2k-4)\lbd^2\\
		&-(3k^2-11k+5)\lbd-2(k^2-4k+2),
	\end{aligned}
	& \text{if } n=2k.
\end{cases}$$

Let $h(\lbd)=g(\lbd)-f(\lbd)$. We now consider the following two cases:

\noindent\textbf{Case 1.} Let $n$ be odd and $n=2k-1$ for some integer $k\geq 7$. Then 
\begin{align*}
	h(\lbd)
	=&\left(r^2-(2k-8)r+k^2-8k+16\right)\lbd^3-\left(3r^2-(6k-22)r+3k^2-22k+40\right)\lbd\\&-2\left(r^2-(2k-7)r+k^2-7k+12\right)\\
	=&(r-(k-4))^2\lbd^3-(3(r-(k-4))^2-2(r-(k-4)))\lbd-2((r-(k-4))^2-(r-(k-4)))\\
	=&t^2\lbd^3-(3t^2-2t)\lbd-2(t^2-t),
\end{align*}
where $t=r-(k-4)$. Thus, $h(\lbd)=t^2(\lbd^3-3\lbd-2)+2t(\lbd+1)=t(\lbd+1)(t(\lbd-2)(\lbd+1)+2)$. Note that $t\in\{-(k-5),-(k-6),\ldots,-1,0,1,\ldots, k-4\}$. If $t=0$, i.e., $r=k-4$, then $h(\lbd)=0$ and the result follows. Otherwise, the largest root of $h(\lbd)$ is the largest root of the polynomial $\lbd^2-\lbd-2+\frac{2}{t}$, i.e, $$\frac{1+\sqrt{9-\frac{8}{t}}}{2}.$$

If $t\geq 1$, then the largest root of $h(\lbd)$ lies in $[1,2)$. Since $k\geq 7$, we have either $r+2\geq 4$ or $2k-r-6\geq 4$. Then the largest root of $f(\lbd)$ is $\rho(G_5(r,2k-7-r))>2$. By Lemma \ref{poly}(i), we have $\rho(G_5(k-4,k-3))< \rho(G_5(r,2k-7-r))$.

If $t\leq -1$, then assume $t=-s$. Note that $s\in \{1,2,\ldots,k-5\}$. If $k\leq 8$, then by direct computation, we can show that $\rho(G_5(k-4,k-3))< \rho(G_5(r,2k-7-r))$. So, we assume that $k\geq 9$. A careful computation implies that either $r+2\geq 7$ or $2k-6-r\geq 7$. So, $\rho(G_5(r,2k-7-r))\geq \sqrt{7}$. Since $$\frac{1+\sqrt{9-\frac{8}{t}}}{2}=\frac{1+\sqrt{9+\frac{8}{s}}}{2}\leq \frac{1+\sqrt{17}}{2}<\sqrt{7},$$ we have $h(\lbd)>0$ for $\lbd\geq \rho(G_5(r,2k-7-r))$. By Lemma \ref{poly}(i), we have $\rho(G_5(k-4,k-3))< \rho(G_5(r,2k-7-r))$.

\medskip
\noindent\textbf{Case 2.} $n$ is even and $n=2k$ for some integer $k\geq 6$. Then 
\begin{align*}
	h(\lbd)
	=&\left(r^2-(2k-7)r+k^2-7k+12\right)\lbd^3-\left(3r^2-(6k-19)r+3k^2-19k+28\right)\lbd\\&-2\left(r^2-(2k-6)r+k^2-6k+8\right).
\end{align*}
The rest of the proof is similar to Case 1.
\qe

\begin{table}[ht]
\centering
	\begin{tabular}{ |c|c|c|c| } 
		\hline
		$n$ & $\rho\left(G_5\left(\left\lfloor\frac{n-7}{2}\right\rfloor,\left\lceil\frac{n-5}{2}\right\rceil\right)\right)\approx$ & $\rho\left(G_6\left(\left\lfloor\frac{n-5}{2}\right\rfloor,\left\lceil\frac{n-5}{2}\right\rceil\right)\right)\approx$ & $ \rho\left(U_{n,3}^*\right)\approx$ \\[5pt]
		\hline
		12 & 2.47283 & 2.59649 & 2.44949 \\ 
		13 & 2.57011 & 2.64261 & 2.52812 \\ 
		14 & 2.63371 & 2.71749 & 2.59205 \\ 
		15 & 2.72676 & 2.76618 & 2.64575 \\
		16 & 2.80058 & 2.84904 & 2.71871 \\
		17 & 2.88644 & 2.89812 & 2.77825 \\
		18 & 2.96471 & 2.98631 & 2.82843 \\
		19 & 3.04378 & 3.03423 & 2.89679 \\
		20 & 3.12337 & 3.12557 & 2.95275 \\
		21 & 3.19668 & 3.17148 & 3 \\
		22 & 3.27601 & 3.26438 & 3.06454 \\
		23 & 3.34451 & 3.30788 & 3.11748 \\
		24 & 3.42279 & 3.40128 & 3.16228 \\
		25 & 3.48726 & 3.44229 & 3.22357 \\
		\hline
	\end{tabular}
\caption{Approximate value of the spectral radius of $G_5\left(\left\lfloor\frac{n-7}{2}\right\rfloor,\left\lceil\frac{n-5}{2}\right\rceil\right)$, $G_6\left(\left\lfloor\frac{n-5}{2}\right\rfloor,\left\lceil\frac{n-5}{2}\right\rceil\right)$ and $U_{n,3}^*$ for $12 \le n \le 25$.}
\label{tab:table}
\end{table}

Next, we compare $\rho\left(G_5\left(\left\lfloor\frac{n-7}{2}\right\rfloor,\left\lceil\frac{n-5}{2}\right\rceil\right)\right)$ and $\rho\left(G_6\left(\left\lfloor\frac{n-5}{2}\right\rfloor,\left\lceil\frac{n-5}{2}\right\rceil\right)\right)$. For $12\leq n\leq 24$, the comparison can be carried out by directly computing the spectral radii of the corresponding graphs, or by using mathematical software (for reference see Table~\ref{tab:table}). The computations show that if $n=12,13,14,15,16,17,18,20,$ then 
$$\rho\left(G_5\left(\left\lfloor\frac{n-7}{2}\right\rfloor,\left\lceil\frac{n-5}{2}\right\rceil\right)\right) <\rho\left(G_6\left(\left\lfloor\frac{n-5}{2}\right\rfloor,\left\lceil\frac{n-5}{2}\right\rceil\right)\right).$$
On the other hand, if $n=19,21,22,23, 24,$ then $$\rho\left(G_5\left(\left\lfloor\frac{n-7}{2}\right\rfloor,\left\lceil\frac{n-5}{2}\right\rceil\right)\right) >\rho\left(G_6\left(\left\lfloor\frac{n-5}{2}\right\rfloor,\left\lceil\frac{n-5}{2}\right\rceil\right)\right).$$ For the remaining values of $n$, the following result establishes the comparison between the spectral radii of $G_5(\lfloor\frac{n-7}{2}\rfloor,\lceil\frac{n-5}{2}\rceil)$ and $G_6(\lfloor\frac{n-5}{2}\rfloor,\lceil\frac{n-5}{2}\rceil)$.

\begin{lem}\label{lem:3-cycle-comp}
	Let $n\geq 25$ be a positive integer. Then the following assertions hold:
	\bdsc
	\item {\rm (i)} If $n$ is odd and $n=2k-1$ for $k\geq 12$, then $\rho(G_5(k-4,k-3))>\rho(G_6(k-3,k-3))$.
	\item {\rm (ii)} If $n$ is even and $n=2k$ for $k\geq 13$, then $\rho(G_5(k-4,k-2))>\rho(G_6(k-3,k-2))$.
	\edsc
\end{lem}

\pf (i) From Lemma \ref{lem:U-19-min}, we have $\phi_\lbd\left(G_5\left(k-4,k-3\right)\right)=\lbd^{n-7}f(\lbd)$, where 
\begin{align*}
	f(\lbd)=&\lbd^7-(2k-1)\lbd^5-2\lbd^4+(k^2+2k-11)\lbd^3+2(2k-5)\lbd^2-(3k^2-14k+13)\lbd\\
	&-2(k^2-5k+5).
\end{align*}

One can calculate that $\phi_\lbd\left(G_6(k-3,k-3)\right)=\lbd^{n-7}g(\lbd)$, where 
\begin{align*}
	g(\lbd)=&\lbd^7-(2k-1)\lbd^5-2\lbd^4+(k^2+2k-12)\lbd^3+2(2k-6)\lbd^2-(3k^2-16k+21)\lbd\\&-2(k-3)^2.
\end{align*}

Let $l=\sqrt{k-1}$. Then $l\geq \sqrt{11}$. Observe that $f(l)=-2(l+1)<0$. Thus, $\rho\left(G_5(k-4,k-3)\right)>l$. On the other hand, it can be checked that $$g(l)=(l+2)(l-1+\sqrt{5})(l-1-\sqrt{5})>0.$$ If $\lbd\geq l$, then we have 
\begin{align*}
	g''(\lbd)&=42\lbd^5-20(2l^2+1)\lbd^3-24\lbd^2+6(l^4+4l^2-9)\lbd+8l^2-16\\
	&=40\lbd^3(\lbd^2-l^2)+\lbd(\lbd^4-20\lbd^2+6l^4+24l^2-154)+\lbd^5-24\lbd^2+8l^2-16\\
	&= 40\lbd^3(\lbd^2-l^2)+\lbd((\lbd^2-10)^2+6(l^2+2)^2-178)+\lbd^2 (\lbd^3-24)+8l^2-16>0.
\end{align*}
Hence $g'(\lbd)$ is a strictly increasing function for $\lbd\geq l$. Since $g'(l)=4l^4-17l^2-16l-8=(2l^2+l+4)(2l^2-l-12)+40>0$, we have $g(\lbd)\geq g(l)$ for $\lbd\geq l$. It follows that $\rho\left(G_6(k-4,k-3)\right)<l.$

(ii) Notice that in this case
\begin{align*}\frac{1}{\lbd^{n-7}}\phi_\lbd\left(G_5\left(k-4,k-2\right)\right)=f(\lbd)=&\lbd^7-2k\lbd^5-2\lbd^4+(k^2+3k-10)\lbd^3+2(2k-4)\lbd^2\\&-(3k^2-11k+5)\lbd-2(k^2-4k+2)
\end{align*}
and
\begin{align*}
	\frac{1}{\lbd^{n-7}}\phi_\lbd\left(G_6\left(k-3,k-2\right)\right)=g(\lbd)=&\lbd^7-2k\lbd^5-2\lbd^4+(k^2+3k-11)\lbd^3+2(2k-5)\lbd^2\\&-(3k^2-13k+13)\lbd-2(k^2-5k+6).
\end{align*}

Using a similar approach as in (i), we can prove that $\rho(G_5(k-4,k-2))>\sqrt{k-\frac{1}{2}}>\rho(G_6(k-3,k-2))$.
\qe

Let $n\geq 25$ be a positive integer. Consider a graph $G_7$ obtained from a path $P_6=[v_1,v_2,v_3,v_4, v_5,v_6]$ and then adding the edge $v_2v_5$ in it. Let $r$ and $s$ be two positive integers and $G_7(r,s)$ be the graph obtained from $G_7$ by attaching $r$ pendant vertices to the vertex $v_1$ and $s$ pendant vertices to the vertex $v_6$. In the next result, we provide a lower bound on the spectral radius of the graph $G_7\left(\left\lfloor\frac{n-5}{2}\right\rfloor-1,\left\lceil\frac{n-5}{2}\right\rceil\right)$, which will be used later. 


\bl\label{cycle_3to4_bound}
Let $n\geq 25$ be a positive integer. Then the following assertions hold:
\bdsc
\item {\rm (i)} If $n$ is even and $n=2k$, then $k-1<\rho(G_7(k-4,k-2))^2$.
\item {\rm (i)} If $n$ is odd and $n=2k+1$, then $k-1<\rho(G_7(k-3,k-2))^2$.
\edsc
\el

\pf To prove (i), let $n$ be even and $n=2k$. It is easy to observe that the square of the spectral radius of $G_7(k-4,k-2)$ is the largest root of the polynomial $$f(\lbd)=\lbd^3-2k\lbd^2+(k^2+4k-17)\lbd-(2k-3)(2k-7).$$
A direct calculation yields $f(k-1)=-5<0$. Hence, the result follows.

The proof of (ii) is similar to that of (i).\qe

Let $n\geq 25$ be a positive integer. Suppose that $n=3k+l$, where $k$ is an integer and $l\in \{-1,0,1\}$. We now compare the spectral radius of the graphs $G_7(\lfloor\frac{n-5}{2}\rfloor-1,\lceil\frac{n-5}{2}\rceil)$ and $G_2(k,k-4+l,k-2)$. For $n\geq 25$, the following result is an immediate consequence of Lemmas \ref{cycle_3to4_bound} and \ref{cycle_4_bound}.

\begin{lem}\label{lem:G4-G2}
	Let $n\geq 25$ be a positive integer such that $n=3k+l$, where $k$ is positive integer and $l\in \{-1,0,1\}$. Then $\rho(G_7(\lfloor\frac{n-5}{2}\rfloor-1,\lceil\frac{n-5}{2}\rceil))>\rho(G_2(k,k-4+l,k-2))$.
\end{lem}
\pf The proof follows from Lemmas \ref{cycle_3to4_bound} and \ref{cycle_4_bound}.\qe

The following results are important to proceed further.
\bt\label{thm:G6} Let $r$ and $n$ be two positive integers such that $n\geq 12$ and $r\leq n-6$. Then $\rho(G_6(r,n-5-r))>\rho(U_{n,3}^*)$.
\et
\pf Note that by Lemma \ref{lem-pendant}, we have $\rho(G_6(r,n-5-r))\geq \rho\left(G_6\left(\left\lfloor\frac{n-5}{2}\right\rfloor,\left\lceil\frac{n-5}{2}\right\rceil\right)\right)$ with equality if and only if $r=\left\lfloor\frac{n-5}{2}\right\rfloor$. If $12\leq n\leq 24$, then it follows from Table \ref{tab:table} that $\rho(G_6(r,n-5-r))\geq \rho\left(G_6\left(\left\lfloor\frac{n-5}{2}\right\rfloor,\left\lceil\frac{n-5}{2}\right\rceil\right)\right)>\rho(U_{n,3}^*)$.

Now, assume that $n\geq 25$. Then from Lemma \ref{lem-sub}, it follows that $$\rho\left(G_6\left(\left\lfloor\frac{n-5}{2}\right\rfloor,\left\lceil\frac{n-5}{2}\right\rceil\right)\right)>\rho\left(G_7\left(\left\lfloor\frac{n-5}{2}\right\rfloor-1,\left\lceil\frac{n-5}{2}\right\rceil\right)\right).$$ By Lemma \ref{lem:G4-G2}, we have the required result.\qe

\bt\label{thm:G5} Let $r$ and $n$ be two positive integers such that $n\geq 12$ and $r\leq n-7$. Then $\rho(G_5(r,n-6-r))>\rho(U_{n,3}^*)$.
\et
\pf Note that by Lemma \ref{lem:U-19-min}, we have $\rho(G_5(r,n-5-r))\geq \rho\left(G_5\left(\left\lfloor\frac{n-7}{2}\right\rfloor,\left\lceil\frac{n-5}{2}\right\rceil\right)\right)$ with equality if and only if $r=\left\lfloor\frac{n-7}{2}\right\rfloor$. If $12\leq n\leq 24$, then it follows from Table \ref{tab:table} that $\rho(G_5(r,n-5-r))>\rho(U_{n,3}^*)$.

Now, assume that $n\geq 25$. By Lemmas \ref{lem-sub} and \ref{lem:3-cycle-comp}, we have 
\begin{align*}
    \rho\left(G_5\left(\left\lfloor\frac{n-7}{2}\right\rfloor,\left\lceil\frac{n-5}{2}\right\rceil\right)\right)&>\rho\left(G_6\left(\left\lfloor\frac{n-5}{2}\right\rfloor,\left\lceil\frac{n-5}{2}\right\rceil\right)\right)\\
    &>\rho\left(G_7\left(\left\lfloor\frac{n-5}{2}\right\rfloor-1,\left\lceil\frac{n-5}{2}\right\rceil\right)\right).
\end{align*} Thus, Lemma \ref{lem:G4-G2} implies the required result.\qe

Next, we provide the main result of this subsection that compares the spectral radius of a graph $G$ in $U_{n,3}$ with girth $3$ is greater than the spectral radius of $U_{n,3}^*$.

\bt\label{thm-3-cycle}
    Let $n \ge 12$ be an integer. If $G \in \mathcal{U}_{n,3}$ has a cycle of length $3$, then $\rho(G) > \rho(U_{n,3}^*)$.
\et

\begin{proof}
   Let us consider $G \in \mathcal{U}_{n,3}$ to be a graph which contains a $3$-cycle $C=[v_1,v_2,v_3,v_1]$ (say). Let $\mathcal{M}$ be the collection of all maximal matchings of $G$. Next, based on the maximal matching of $G$, we divide into the following cases:

   \begin{itemize}
       \item[Case 1.] For all $\Gamma \in \mathcal{M}$, every edge in $\Gamma$ has at least one endpoint on the $3$-cycle. Then the only possible structure of $G$ is as follows: A $3$-cycle, with pendant edges attached to each vertex of the cycle (see Figure~\ref{fig:3-1}). Let $G^*$ be the graph obtained from $G$ by subdividing the edge $v_1v_2$. Then using Lemma~\ref{lem-sub}, we have $\rho(G^*) < \rho(G)$, but $G^*$ is a graph on $n + 1$ vertices with a $4$-cycle. Next, let $G^{**}$ be the subgraph of $G^*$ obtained by deleting one pendant vertex of $G^*$, then $G^{**} \in \mathcal{U}_{n,3}$ and using Lemma~\ref{lem:subgraph} and Theorem \ref{main_1}, we have $\rho(U_{n,3}^*)\leq\rho(G^{**}) < \rho(G^*) < \rho(G)$. 

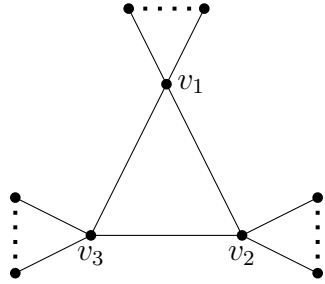
\begin{figure}[ht]
    \centering
\begin{tikzpicture}[scale=1]
  \coordinate (A) at (0.00, 0.00);
  \coordinate (B) at (2.00, 0.00);
  \coordinate (C) at (1.00, 2.00);
  \coordinate (D) at (0.50, 3.00);
  \coordinate (E) at (1.50, 3.00);
  \coordinate (F) at (3.00, 0.50);
  \coordinate (G) at (3.00, -0.50);
  \coordinate (H) at (-1.00, 0.50);
  \coordinate (I) at (-1.00, -0.50);

  \draw (0,0) node[anchor=north] {$v_3$};
  \draw (2,0) node[anchor=north] {$v_2$};
  \draw (1,2) node[anchor=west] {$v_1$};

  \draw (A) -- (B);
  \draw (A) -- (C);
  \draw (C) -- (B);
  \draw (C) -- (D);
  \draw (C) -- (E);
  \draw (B) -- (F);
  \draw (B) -- (G);
  \draw (H) -- (A);
  \draw (I) -- (A);
  \fill[black] (A) circle (2pt);
  \fill[black] (B) circle (2pt);
  \fill[black] (C) circle (2pt);
  \fill[black] (D) circle (2pt);
  \fill[black] (E) circle (2pt);
  \fill[black] (F) circle (2pt);
  \fill[black] (G) circle (2pt);
  \fill[black] (H) circle (2pt);
  \fill[black] (I) circle (2pt);

  \draw[loosely dotted, line width=1.5pt] (3, -0.50) -- (3, 0.50);
  \draw[loosely dotted, line width=1.5pt] (-1, -0.50) -- (-1, 0.50);
  \draw[loosely dotted, line width=1.5pt] (0.50, 3.00) -- (1.50, 3.00);
\end{tikzpicture}
    \caption{Unicyclic graph with a $3$-cycle and $\gamma = 3$.}
    \label{fig:3-1}
\end{figure}
       \item[Case 2.] For all $\Gamma \in \mathcal{M}$, there exists at least one edge in $\Gamma$ such that both of its endpoints are not part of the $3$-cycle. Let the edge be $uv$. Then $G-u-v$ consists of a graph containing a $3$-cycle, possibly together with some isolated vertices ($K_1$'s) and has $\gamma=2$. Thus, it follows from Theorem~\ref{thm-gamma-2} that the nontrivial component of $G-u-v$ is either $U_1$ or $U_2$ (see Figure~\ref{mat_2_cycle_3}). Without loss of generality, we assume that the vertex $u$ is closer to the vertices of the $3$-cycle than $v$.
\begin{itemize}
    \item[Subcase 2.1.] Suppose that the nontrivial component of $G-u-v$ is $U_1$. Since $\gamma=3$, we have $\min\{ d(u,v_1),d(u,v_2),d(u,v_3)\} \le 2$. Thus, $G$ is isomorphic to one of the graphs $U_{12}$, $U_{13}$, or $U_{14}$ (see Figure~\ref{fig:3-2-1}). In the graphs $U_{12}, U_{13}$, and $U_{14}$, one of $v_1$ or $v_2$ may have no pendant neighbors, but at least one of them has at least one pendant neighbor.
    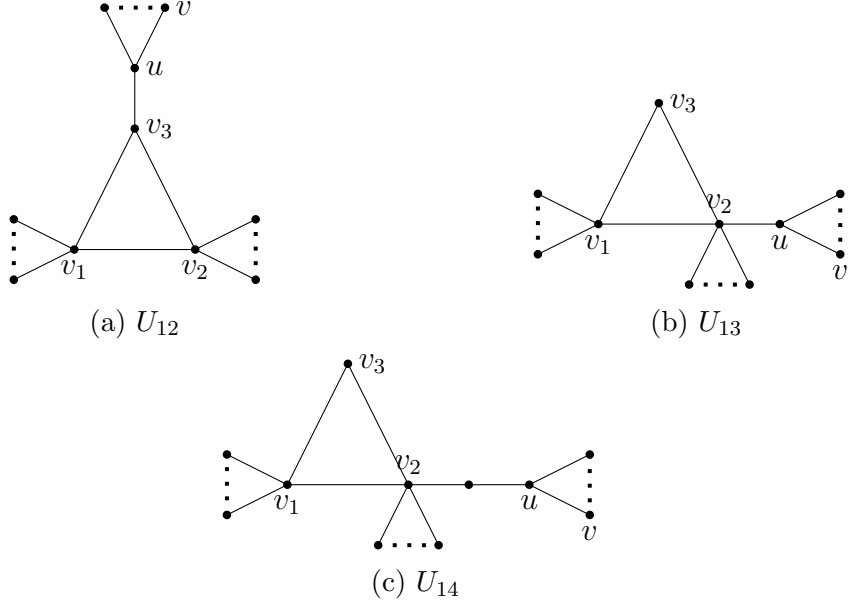
\begin{figure}[ht]
    \centering
    \begin{subfigure}[b]{0.45\textwidth}
        \centering
        \begin{tikzpicture}[scale=0.8] 
          \coordinate (A) at (0.00, 0.00);
          \coordinate (B) at (2.00, 0.00);
          \coordinate (C) at (1.00, 2.00);
          \coordinate (F) at (3.00, 0.50);
          \coordinate (G) at (3.00, -0.50);
          \coordinate (H) at (-1.00, 0.50);
          \coordinate (I) at (-1.00, -0.50);
          \coordinate (D) at (1.00, 3.00);
          \coordinate (E) at (0.50, 4.00);
          \coordinate (J) at (1.50, 4.00);
        
          \draw (0,0) node[anchor=north] {$v_1$};
          \draw (2,0) node[anchor=north] {$v_2$};
          \draw (1,2) node[anchor=west] {$v_3$};
          \draw (1,3) node[anchor=west] {$u$};
          \draw (1.5,4) node[anchor=west] {$v$};
        
          \draw (A) -- (B); \draw (A) -- (C); \draw (C) -- (B);
          \draw (B) -- (F); \draw (B) -- (G); \draw (H) -- (A);
          \draw (I) -- (A); \draw (C) -- (D); \draw (D) -- (E); \draw (D) -- (J);
          \foreach \p in {A,B,C,F,G,H,I,D,E,J} \fill[black] (\p) circle (2pt);

          \draw[loosely dotted, line width=1.5pt] (3, -0.50) -- (3, 0.50);
          \draw[loosely dotted, line width=1.5pt] (-1, -0.50) -- (-1, 0.50);
          \draw[loosely dotted, line width=1.5pt] (0.50, 4.00) -- (1.50, 4.00);
        \end{tikzpicture}
        \caption{$U_{12}$}
    \end{subfigure}
    \begin{subfigure}[b]{0.45\textwidth}
        \centering
        \begin{tikzpicture}[scale=0.8]
          \coordinate (A) at (0.00, 0.00);
          \coordinate (B) at (2.00, 0.00);
          \coordinate (C) at (1.00, 2.00);
          \coordinate (H) at (-1.00, 0.50);
          \coordinate (I) at (-1.00, -0.50);
          \coordinate (D) at (1.50, -1.00);
          \coordinate (E) at (2.50, -1.00);
          \coordinate (F) at (3.00, 0.00);
          \coordinate (G) at (4.00, 0.50);
          \coordinate (J) at (4.00, -0.50);
        
          \draw (0,0) node[anchor=north] {$v_1$};
          \draw (2,0) node[anchor=south] {$v_2$};
          \draw (1,2) node[anchor=west] {$v_3$};
          \draw (3,0) node[anchor=north] {$u$};
          \draw (4,-0.5) node[anchor=north] {$v$};
        
          \draw (A) -- (B); \draw (A) -- (C); \draw (C) -- (B);
          \draw (H) -- (A); \draw (I) -- (A); \draw (B) -- (D);
          \draw (B) -- (E); \draw (B) -- (F); \draw (F) -- (G); \draw (F) -- (J);
          \foreach \p in {A,B,C,H,I,D,E,F,G,J} \fill[black] (\p) circle (2pt);

          \draw[loosely dotted, line width=1.5pt] (1.50, -1.00) -- (2.50, -1.00);
          \draw[loosely dotted, line width=1.5pt] (-1, -0.50) -- (-1, 0.50);
          \draw[loosely dotted, line width=1.5pt] (4.00, 0.50) -- (4.00, -0.50);
        \end{tikzpicture}
        \caption{$U_{13}$}
    \end{subfigure}
    \begin{subfigure}[b]{0.45\textwidth}
        \centering
        \begin{tikzpicture}[scale=0.8]
          \coordinate (A) at (0.00, 0.00);
          \coordinate (B) at (2.00, 0.00);
          \coordinate (C) at (1.00, 2.00);
          \coordinate (H) at (-1.00, 0.50);
          \coordinate (I) at (-1.00, -0.50);
          \coordinate (D) at (1.50, -1.00);
          \coordinate (E) at (2.50, -1.00);
          \coordinate (F) at (3.00, 0.00);
          \coordinate (G) at (4.00, 0.00);
          \coordinate (J) at (5.00, 0.50);
          \coordinate (K) at (5.00, -0.50);
        
          \draw (0,0) node[anchor=north] {$v_1$};
          \draw (2,0) node[anchor=south] {$v_2$};
          \draw (1,2) node[anchor=west] {$v_3$};
          \draw (5,-0.5) node[anchor=north] {$v$};
          \draw (4,0) node[anchor=north] {$u$};
        
          \draw (A) -- (B); \draw (A) -- (C); \draw (C) -- (B);
          \draw (H) -- (A); \draw (I) -- (A); \draw (B) -- (D);
          \draw (B) -- (E); \draw (B) -- (F); \draw (F) -- (G); \draw (G) -- (J); \draw (G) -- (K);
          \foreach \p in {A,B,C,H,I,D,E,F,G,J,K} \fill[black] (\p) circle (2pt);
          
          \draw[loosely dotted, line width=1.5pt] (1.50, -1.00) -- (2.50, -1.00);
          \draw[loosely dotted, line width=1.5pt] (-1, -0.50) -- (-1, 0.50);
          \draw[loosely dotted, line width=1.5pt] (5.00, 0.50) -- (5.00, -0.50);
        \end{tikzpicture}
        \caption{$U_{14}$}
    \end{subfigure}
    \caption{Unicyclic graphs with a $3$-cycle and $\gamma = 3$.}
    \label{fig:3-2-1}
\end{figure}

    Note that, in all the structures in Figure~\ref{fig:3-2-1}, if we subdivide the edge $v_1v_2$ and then take a subgraph by deleting a pendant vertex, then we get a graph which has a $4$-cycle, belongs to $\mathcal{U}_{n,3}$ and has spectral radius less than that of $G$. Hence, the result follows from Theorem \ref{main_1}.
    
    \item[Subcase 2.2.] Suppose that the nontrivial component of $G-u-v$ is $U_2$. Since $\gamma=3$, we have $\min\{ d(u,v_1),d(u,v_2),d(u,v_3)\} \le 2$. Thus, $G$ is isomorphic to one of the graphs $U_{15}, U_{16}, U_{17}$, and $U_{18}$ (see Figure~\ref{fig:3-2-3}).
    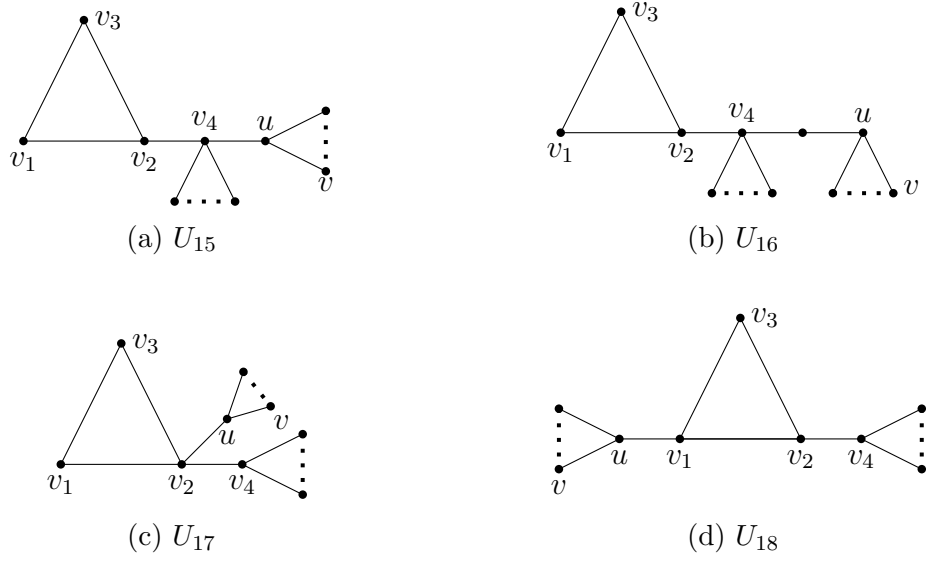
\begin{figure}[ht]
    \centering
    \begin{subfigure}[b]{0.45\textwidth}
        \centering
        \begin{tikzpicture}[scale=0.8]
  \coordinate (A) at (0.00, 0.00);
  \coordinate (B) at (2.00, 0.00);
  \coordinate (C) at (1.00, 2.00);
  \coordinate (F) at (3.00, 0.00);
  \coordinate (D) at (4.00, 0.00);
  \coordinate (E) at (5.00, 0.50);
  \coordinate (I) at (5.00, -0.50);
  \coordinate (G) at (2.50, -1.00);
  \coordinate (H) at (3.50, -1.00);

    \draw (0,0) node[anchor=north] {$v_1$};
    \draw (2,0) node[anchor=north] {$v_2$};
    \draw (1,2) node[anchor=west] {$v_3$};
    \draw (3,0) node[anchor=south] {$v_4$};
    \draw (4,0) node[anchor=south] {$u$};
    \draw (5,-1) node[anchor=south] {$v$};

  \draw (A) -- (B);
  \draw (A) -- (C);
  \draw (C) -- (B);
  \draw (B) -- (F);
  \draw (F) -- (D);
  \draw (D) -- (E);
  \draw (D) -- (I);
  \draw (F) -- (G);
  \draw (F) -- (H);
  \fill[black] (A) circle (2pt);
  \fill[black] (B) circle (2pt);
  \fill[black] (C) circle (2pt);
  \fill[black] (F) circle (2pt);
  \fill[black] (D) circle (2pt);
  \fill[black] (E) circle (2pt);
  \fill[black] (I) circle (2pt);
  \fill[black] (G) circle (2pt);
  \fill[black] (H) circle (2pt);

  \draw[loosely dotted, line width=1.5pt] (3.50, -1.00) -- (2.50, -1.00);
  \draw[loosely dotted, line width=1.5pt] (5.00, 0.50) -- (5.00, -0.50);
\end{tikzpicture}
        \caption{$U_{15}$}
    \end{subfigure}
    \begin{subfigure}[b]{0.45\textwidth}
        \centering
        \begin{tikzpicture}[scale=0.8]
  \coordinate (A) at (0.00, 0.00);
  \coordinate (B) at (2.00, 0.00);
  \coordinate (C) at (1.00, 2.00);
  \coordinate (F) at (3.00, 0.00);
  \coordinate (D) at (4.00, 0.00);
  \coordinate (G) at (2.50, -1.00);
  \coordinate (H) at (3.50, -1.00);
  \coordinate (E) at (5.00, 0.00);
  \coordinate (I) at (4.50, -1.00);
  \coordinate (J) at (5.50, -1.00);

    \draw (0,0) node[anchor=north] {$v_1$};
    \draw (2,0) node[anchor=north] {$v_2$};
    \draw (1,2) node[anchor=west] {$v_3$};
    \draw (3,0) node[anchor=south] {$v_4$};
    \draw (5,0) node[anchor=south] {$u$};
    \draw (5.8,-1.2) node[anchor=south] {$v$};

  \draw (A) -- (B);
  \draw (A) -- (C);
  \draw (C) -- (B);
  \draw (B) -- (F);
  \draw (F) -- (D);
  \draw (F) -- (G);
  \draw (F) -- (H);
  \draw (D) -- (E);
  \draw (E) -- (I);
  \draw (E) -- (J);
  \fill[black] (A) circle (2pt);
  \fill[black] (B) circle (2pt);
  \fill[black] (C) circle (2pt);
  \fill[black] (F) circle (2pt);
  \fill[black] (D) circle (2pt);
  \fill[black] (G) circle (2pt);
  \fill[black] (H) circle (2pt);
  \fill[black] (E) circle (2pt);
  \fill[black] (I) circle (2pt);
  \fill[black] (J) circle (2pt);

  \draw[loosely dotted, line width=1.5pt] (3.50, -1.00) -- (2.50, -1.00);
  \draw[loosely dotted, line width=1.5pt] (4.50, -1.00) -- (5.50, -1.00);
\end{tikzpicture}
        \caption{$U_{16}$}
    \end{subfigure}

\vspace{0.5cm} 

\begin{subfigure}[b]{0.45\textwidth}
        \centering
        \begin{tikzpicture}[scale=0.8]
  \coordinate (A) at (0.00, 0.00);
  \coordinate (B) at (2.00, 0.00);
  \coordinate (C) at (1.00, 2.00);
  \coordinate (F) at (3.00, 0.00);
  \coordinate (G) at (4.00, 0.50);
  \coordinate (H) at (4.00, -0.50);
  \coordinate (D) at (2.75, 0.75);
  \coordinate (E) at (3.02, 1.53);
  \coordinate (I) at (3.47, 0.96);

    \draw (0,0) node[anchor=north] {$v_1$};
    \draw (2,0) node[anchor=north] {$v_2$};
    \draw (1,2) node[anchor=west] {$v_3$};
    \draw (3,0) node[anchor=north] {$v_4$};
    \draw (2.75, 0.75) node[anchor=north] {$u$};
    \draw (3.65, 1) node[anchor=north] {$v$};

  \draw (A) -- (B);
  \draw (A) -- (C);
  \draw (C) -- (B);
  \draw (B) -- (F);
  \draw (B) -- (D);
  \draw (F) -- (G);
  \draw (F) -- (H);
  \draw (D) -- (E);
  \draw (D) -- (I);
  \fill[black] (A) circle (2pt);
  \fill[black] (B) circle (2pt);
  \fill[black] (C) circle (2pt);
  \fill[black] (F) circle (2pt);
  \fill[black] (G) circle (2pt);
  \fill[black] (H) circle (2pt);
  \fill[black] (D) circle (2pt);
  \fill[black] (E) circle (2pt);
  \fill[black] (I) circle (2pt);

  \draw[loosely dotted, line width=1.5pt] (3.02, 1.53) -- (3.47, 0.96);
  \draw[loosely dotted, line width=1.5pt] (4.00, 0.50) -- (4.00, -0.50);
\end{tikzpicture}
        \caption{$U_{17}$}
    \end{subfigure}
    \begin{subfigure}[b]{0.45\textwidth}
        \centering
        \begin{tikzpicture}[scale=0.8]
  \coordinate (A) at (0.00, 0.00);
  \coordinate (B) at (2.00, 0.00);
  \coordinate (C) at (1.00, 2.00);
  \coordinate (F) at (3.00, 0.00);
  \coordinate (G) at (4.00, 0.50);
  \coordinate (H) at (4.00, -0.50);
  \coordinate (D) at (-1.00, 0.00);
  \coordinate (I) at (-2.00, 0.50);
  \coordinate (J) at (-2.00, -0.50);

    \draw (0,0) node[anchor=north] {$v_1$};
    \draw (2,0) node[anchor=north] {$v_2$};
    \draw (1,2) node[anchor=west] {$v_3$};
    \draw (3,0) node[anchor=north] {$v_4$};
    \draw (-1,0) node[anchor=north] {$u$};
    \draw (-2,-0.5) node[anchor=north] {$v$};

  \draw (A) -- (B);
  \draw (A) -- (C);
  \draw (C) -- (B);
  \draw (B) -- (F);
  \draw (B) -- (D);
  \draw (F) -- (G);
  \draw (F) -- (H);
  \draw (I) -- (D);
  \draw (J) -- (D);
  \fill[black] (A) circle (2pt);
  \fill[black] (B) circle (2pt);
  \fill[black] (C) circle (2pt);
  \fill[black] (F) circle (2pt);
  \fill[black] (G) circle (2pt);
  \fill[black] (H) circle (2pt);
  \fill[black] (D) circle (2pt);
  \fill[black] (I) circle (2pt);
  \fill[black] (J) circle (2pt);

  \draw[loosely dotted, line width=1.5pt] (-2.00, 0.50) -- (-2.00, -0.50);
  \draw[loosely dotted, line width=1.5pt] (4.00, 0.50) -- (4.00, -0.50);
\end{tikzpicture}
        \caption{$U_{18}$}
    \end{subfigure}
    \caption{Unicyclic graphs with a $3$-cycle and $\gamma = 3$.}
    \label{fig:3-2-3}
\end{figure}

\begin{itemize}
    \item[(a)] Suppose $G \cong U_{15}$ (see (a) of Figure~\ref{fig:3-2-3}). Let $G^{*}$ be the graph obtained from $G$ by subdividing the edge $ uv_4$, then using Lemma~\ref{lem-sub}, we have $\rho(G^*) < \rho(G)$. Next, let $G^{**}$ be the graph obtained from $G^*$ by deleting a pendant vertex. Since $G^{**}$ is a proper subgraph of $G^*$, using Lemma~\ref{lem:subgraph}, we have that $\rho(G^{**}) < \rho(G^*) < \rho(G)$. Observe that $G^{**} \cong U_{16}$. Thus, we have $\rho(U_{15})>\rho(U_{16})$. Note that $U_{16}\cong G_5(r,n-6-r)$ for some positive integer $r\leq n-7$. Thus, from Theorem \ref{thm:G5} we get $\rho(U_{15})> \rho(U_{16})>\rho(U_{n,3}^*)$.
    \item[(b)] Suppose $G \cong U_{18}$ (see (d) of Figure~\ref{fig:3-2-3}). Suppose $x_{v_2} \ge x_{v_1}$, then let $G^{*}$ be the graph obtained from $G$ by deleting the edge $u v_{v_1}$ and adding the edge $uv_2$. On the other hand if $x_{v_1} > x_{v_2}$, then let $G^{*}$ be the graph obtained from $G$ by deleting the edge $v_2v_4$ and adding an edge between $u$ and $v_1$ and between $v_1$ and $v_4$. In both of the cases, using Lemma~\ref{lem:rho}, we have $\rho(G^*) > \rho(G)$ and $G^* \cong U_{17}$ (see (c) of Figure~\ref{fig:3-2-3}). Hence, we have $\rho(U_{17})>\rho(U_{18})$. Note that $U_{18}\cong G_6(r,n-5-r)$ for some positive integer $r\leq n-6$. Thus, from Theorem \ref{thm:G6} we get $\rho(U_{17})> \rho(U_{18})>\rho(U_{n,3}^*)$.
\end{itemize}
\end{itemize}
\end{itemize}

This completes the proof.
\end{proof}


\subsection{Graphs with minimum spectral radius in $\mathcal{U}_{n,3}$}
We are now ready to find the graph that minimizes the spectral radius among all graphs in $\mathcal{U}_{n,3}$. We state and prove the main result of this section that characterizes the graphs with the minimum spectral radius among all unicyclic graphs with matching number $3$.
\begin{theorem}
	Let $G \in \mathcal{U}_{n,3}$ with $n \ge 12$ be a graph with minimum spectral radius. Then
	\[
	G \cong 
	\begin{cases} 
			G_2\left(\frac{n+1}{3}, \frac{n-14}{3}, \frac{n-5}{3}\right) & \text{if } n \equiv -1 \pmod{3} \\
			G_2\left(\frac{n}{3}, \frac{n-12}{3}, \frac{n-6}{3}\right) \text{ or } G_3\left(\frac{n-6}{3}, \frac{n-6}{3}, \frac{n-6}{3}\right)& \text{if } n \equiv 0 \pmod{3}, \text{ and } \\
			G_2\left(\frac{n-1}{3}, \frac{n-10}{3}, \frac{n-7}{3}\right) & \text{if } n \equiv 1 \pmod{3}.
		\end{cases}
	\]
\end{theorem}

\begin{proof}
	Let $G \in \mathcal{U}_{n,3}$ be a graph with minimum spectral radius. Suppose $G$ has a cycle of length $7$, then it cannot contain any more vertices except the vertices of the cycle, otherwise $\gamma > 3$. Since $n \ge 12$, the minimal graph $G$ cannot have a cycle of length $7$ or more. Next, using Theorems \ref{thm-5-cycle} and \ref{thm-3-cycle}, we can conclude that $G$ must contain a cycle of length $4$ or $6$. Finally, the result directly follows from Theorems~\ref{thm-6-cycle} and \ref{main_1}.
\end{proof}

\section*{Declaration of competing interest}

The authors declare that they have no known competing financial interests or personal relationships that could have appeared to influence the work reported in this article.



\section*{Data availability}

Data sharing is not applicable to this article as no datasets were generated or analyzed during the current study.

\small{

\end{document}